\documentclass[a4paper]{amsart}

\usepackage[utf8]{inputenc}
\usepackage[T1]{fontenc}
\usepackage[english]{babel}

\usepackage{amsfonts}
\usepackage{amsmath}
\usepackage{amsrefs}
\usepackage{amssymb}
\usepackage{amstext}
\usepackage{amsthm}

\usepackage{geometry}
\usepackage{fullpage}

\usepackage[shortlabels]{enumitem}
\usepackage{graphicx}
\usepackage{hyperref}
\usepackage{listings}
\usepackage{mathrsfs}
\usepackage{mathtools}
\usepackage{multicol}
\usepackage{shuffle}

\usepackage{tikz}
\usepackage{tikz-cd}
\usetikzlibrary{arrows,backgrounds,chains,decorations.pathreplacing,patterns,shapes}

\newcommand{\area}{\mathsf{area}}
\newcommand{\dinv}{\mathsf{dinv}}
\newcommand{\bounce}{\mathsf{bounce}}

\newcommand{\Val}{\mathsf{Val}}
\newcommand{\Rise}{\mathsf{Rise}}
\newcommand{\Peak}{\mathsf{Peak}}
\newcommand{\ZVal}{\mathsf{ZVal}}
\newcommand{\DRise}{\mathsf{DRise}}
\newcommand{\DPeak}{\mathsf{DPeak}}

\newcommand{\D}{\mathsf{D}} 
\newcommand{\DD}{\mathsf{DD}} 
\newcommand{\DDd}{\mathsf{DDd}} 
\newcommand{\DDb}{\mathsf{DDb}} 

\newcommand{\PLD}{\mathsf{PLD}} 

\newcommand{\RP}{\mathsf{RP}} 

\newcommand{\qbinom}[2]{\genfrac{[}{]}{0pt}{}{#1}{#2}}

\makeatletter

\pgfkeys{
	/tikz/sharp angle/.code={%
		\pgfsetarrowoptions{sharp >}{#1}%
		\pgfsetarrowoptions{sharp <}{-#1}%
	},
	/tikz/sharp > angle/.code={%
		\pgfsetarrowoptions{sharp >}{#1}%
	},
	/tikz/sharp < angle/.code={%
		\pgfsetarrowoptions{sharp <}{#1}%
	},
	/tikz/sharp protrude/.code=\csname if#1\endcsname\qrr@tikz@sharp@z@-0.05\p@\else\qrr@tikz@sharp@z@\z@\fi,
	/tikz/sharp protrude/.default=true
}

\newdimen\qrr@tikz@sharp@z@
\qrr@tikz@sharp@z@\z@
\pgfarrowsdeclare{sharp >}{sharp >}{%
	\edef\pgf@marshal{\noexpand\pgfutil@in@{and}{\pgfgetarrowoptions{sharp >}}}%
	\pgf@marshal
	\ifpgfutil@in@
	\edef\pgf@tempa{\pgfgetarrowoptions{sharp >}}
	\expandafter\qrr@tikz@sharp@parse\pgf@tempa\@qrr@tikz@sharp@parse
	\else
	\qrr@tikz@sharp@parse\pgfgetarrowoptions{sharp >}and-\pgfgetarrowoptions{sharp >}\@qrr@tikz@sharp@parse
	\fi
	\pgfmathparse{max(\pgf@tempa,\pgf@tempb,0)}%
	\let\qrr@tikz@sharp@max\pgfmathresult
	\pgfmathsetlength\pgf@xa{.5*\pgflinewidth * tan(\qrr@tikz@sharp@max)}%
	\pgfarrowsleftextend{+\pgf@xa}%
	\pgfarrowsrightextend{+\pgf@xa}%
}{%
	\edef\pgf@marshal{\noexpand\pgfutil@in@{and}{\pgfgetarrowoptions{sharp >}}}%
	\pgf@marshal
	\ifpgfutil@in@
	\edef\pgf@tempa{\pgfgetarrowoptions{sharp >}}
	\expandafter\qrr@tikz@sharp@parse\pgf@tempa\@qrr@tikz@sharp@parse
	\else
	\qrr@tikz@sharp@parse\pgfgetarrowoptions{sharp >}and-\pgfgetarrowoptions{sharp >}\@qrr@tikz@sharp@parse
	\fi
	\pgfmathsetlength\pgf@ya{.5*\pgflinewidth * tan(max(\pgf@tempa,\pgf@tempb,0))}%
	\pgfmathsetlength\pgf@xa{-.5*\pgflinewidth * tan(\pgf@tempa)}%
	\pgfmathsetlength\pgf@xb{-.5*\pgflinewidth * tan(\pgf@tempb)}%
	\advance\pgf@xa\pgf@ya
	\advance\pgf@xb\pgf@ya
	\ifdim\pgf@xa>\pgf@xb
	\pgftransformyscale{-1}%
	\pgf@xc\pgf@xb
	\pgf@xb\pgf@xa
	\pgf@xa\pgf@xc
	\fi
	\pgfpathmoveto{\pgfqpoint{\qrr@tikz@sharp@z@}{.5\pgflinewidth}}%
	\pgfpathlineto{\pgfqpoint{\pgf@xa}{.5\pgflinewidth}}%
	\pgfpathlineto{\pgfqpoint{\pgf@ya}{+0pt}}%
	\pgfpathlineto{\pgfqpoint{\pgf@xb}{-.5\pgflinewidth}}%
	\pgfpathlineto{\pgfqpoint{\qrr@tikz@sharp@z@}{-.5\pgflinewidth}}%
	\pgfusepathqfill
}
\pgfarrowsdeclare{sharp <}{sharp <}{%
	\edef\pgf@marshal{\noexpand\pgfutil@in@{and}{\pgfgetarrowoptions{sharp <}}}%
	\pgf@marshal
	\ifpgfutil@in@
	\edef\pgf@tempa{\pgfgetarrowoptions{sharp <}}
	\expandafter\qrr@tikz@sharp@parse\pgf@tempa\@qrr@tikz@sharp@parse
	\else
	\expandafter\qrr@tikz@sharp@parse\pgfgetarrowoptions{sharp <}and-\pgfgetarrowoptions{sharp <}\@qrr@tikz@sharp@parse
	\fi
	\pgfmathparse{max(\pgf@tempa,\pgf@tempb,0)}%
	\let\qrr@tikz@sharp@max\pgfmathresult
	\pgfmathsetlength\pgf@xa{.5*\pgflinewidth * tan(\qrr@tikz@sharp@max)}%
	\pgfarrowsleftextend{+\pgf@xa}%
	\pgfarrowsrightextend{+\pgf@xa}%
}{%
	\edef\pgf@marshal{\noexpand\pgfutil@in@{and}{\pgfgetarrowoptions{sharp <}}}%
	\pgf@marshal
	\ifpgfutil@in@
	\edef\pgf@tempa{\pgfgetarrowoptions{sharp <}}
	\expandafter\qrr@tikz@sharp@parse\pgf@tempa\@qrr@tikz@sharp@parse
	\else
	\expandafter\qrr@tikz@sharp@parse\pgfgetarrowoptions{sharp <}and-\pgfgetarrowoptions{sharp <}\@qrr@tikz@sharp@parse
	\fi
	\pgfmathsetlength\pgf@ya{.5*\pgflinewidth * tan(max(\pgf@tempa,\pgf@tempb,0))}%
	\pgfmathsetlength\pgf@xa{-.5*\pgflinewidth * tan(\pgf@tempa)}%
	\pgfmathsetlength\pgf@xb{-.5*\pgflinewidth * tan(\pgf@tempb)}%
	\advance\pgf@xa\pgf@ya
	\advance\pgf@xb\pgf@ya
	\ifdim\pgf@xa>\pgf@xb
	\pgftransformyscale{-1}%
	\pgf@xc\pgf@xb
	\pgf@xb\pgf@xa
	\pgf@xa\pgf@xc
	\fi
	\pgfpathmoveto{\pgfqpoint{\qrr@tikz@sharp@z@}{.5\pgflinewidth}}%
	\pgfpathlineto{\pgfqpoint{\pgf@xa}{.5\pgflinewidth}}%
	\pgfpathlineto{\pgfqpoint{\pgf@ya}{+0pt}}%
	\pgfpathlineto{\pgfqpoint{\pgf@xb}{-.5\pgflinewidth}}%
	\pgfpathlineto{\pgfqpoint{\qrr@tikz@sharp@z@}{-.5\pgflinewidth}}%
	\pgfusepathqfill
}
\def\qrr@tikz@sharp@parse#1and#2\@qrr@tikz@sharp@parse{\def\pgf@tempa{#1}\def\pgf@tempb{#2}}

\makeatother

\newcommand\multiset[2]%
{\mathchoice{\left(\kern-0.4em{\binom{#1}{#2}}\kern-0.4em\right)}
	{\bigl(\kern-0.2em{\binom{#1}{#2}}\kern-0.2em\bigr)}
	{\bigl(\kern-0.2em{\binom{#1}{#2}}\kern-0.2em\bigr)}
	{\bigl(\kern-0.2em{\binom{#1}{#2}}\kern-0.2em\bigr)}}

\let\existstemp\exists \renewcommand*{\exists}{\mathop \existstemp}
\let\foralltemp\forall \renewcommand*{\forall}{\mathop \foralltemp}

\def\quotient#1#2{\raise1ex\hbox{$#1$}\Big/\lower1ex\hbox{$#2$}}

\newcommand{\<}{\langle}
\renewcommand{\>}{\rangle}

\newtheorem{theorem}{Theorem}[section]
\newtheorem{lemma}[theorem]{Lemma}
\newtheorem{proposition}[theorem]{Proposition}
\newtheorem{corollary}[theorem]{Corollary}
\newtheorem{conjecture}[theorem]{Conjecture}

\theoremstyle{definition}
\newtheorem{definition}[theorem]{Definition}

\newtheorem{example}[theorem]{Example}

\theoremstyle{remark}
\newtheorem{remark}[theorem]{Remark}

\title{The Schröder case of the generalized Delta conjecture}

\author{Michele D'Adderio}
\address{Universit\'e Libre de Bruxelles (ULB)\\D\'epartement de Math\'ematique\\ Boulevard du Triomphe, B-1050 Bruxelles\\ Belgium}\email{mdadderi@ulb.ac.be}

\author{Alessandro Iraci}
\address{Universit\'a di Pisa and Universit\'e Libre de Bruxelles (ULB)\\Dipartimento di Matematica\\ Largo Bruno Pontecorvo 5, 56127 Pisa\\ Italia}\email{iraci@student.dm.unipi.it}

\author{Anna Vanden Wyngaerd}
\address{Universit\'e Libre de Bruxelles (ULB)\\D\'epartement de Math\'ematique\\ Boulevard du Triomphe, B-1050 Bruxelles\\ Belgium}\email{anvdwyng@ulb.ac.be}

\begin{document}
	
\begin{abstract}
	We prove the Schr\"{o}der case, i.e. the case $\<\cdot , e_{n-d}h_d\>$, of the conjecture of Haglund Remmel and Wilson \cite{Haglund-Remmel-Wilson-2015} for $\Delta_{h_m}\Delta_{e_{n-k-1}}'e_n$ in terms of decorated partially labelled Dyck paths, which we call \emph{generalized Delta conjecture}. This result extends the Schr\"{o}der case of the Delta conjecture proved in \cite{DAdderio-VandenWyngaerd-2017}, which in turn generalized the $q,t$-Schr\"{o}der of Haglund \cite{Haglund-Schroeder-2004}. The proof gives a recursion for these polynomials that extends the ones known for the aforementioned special cases. Also, we give another combinatorial interpretation of the same polynomial in terms of a new bounce statistic. Moreover, we give two more interpretations of the same polynomial in terms of doubly decorated parallelogram polyominoes, extending some of the results in \cite{DAdderio-Iraci-polyominoes-2017}, which in turn extended results in \cite{Aval-DAdderio-Dukes-Hicks-LeBorgne-2014}. Also, we provide combinatorial bijections explaining some of the equivalences among these interpretations.
\end{abstract}
	
\maketitle
\tableofcontents

\section{Introduction}

In \cite{Haglund-Remmel-Wilson-2015}, Haglund, Remmel and Wilson conjectured a combinatorial formula for $\Delta_{e_{n-k-1}}'e_n$ in terms of decorated labelled Dyck paths, which they called \emph{Delta conjecture}, after the so called delta operators $\Delta_f'$ introduced by Bergeron, Garsia, Haiman, and Tesler \cite{Bergeron-Garsia-Haiman-Tesler-Positivity-1999} for any symmetric function $f$.

The special case $k=0$ gives precisely the \emph{Shuffle conjecture} in \cite{HHLRU-2005}, now a theorem of Carlsson and Mellit \cite{Carlsson-Mellit-ShuffleConj-2015}. The latter turns out to be a combinatorial formula for the Frobenius characteristic of the $\mathfrak{S}_n$-module of diagonal harmonics studied by Garsia and Haiman in relation to the famous \emph{$n!$ conjecture}, now $n!$ theorem of Haiman \cite{Haiman-nfactorial-2001}.

Though the full Delta conjecture is still open, several of its consequences have been recently proved: for example, the specializations at $q=0$ and at $q=1$ have been proved in \cite{Garsia-Haglund-Remmel-Yoo-2017} and \cite{Romero-Deltaq1-2017} respectively. Also, the special cases $\<\cdot , e_{n-d}h_d\>$ and $\<\cdot , h_{n-d}h_d\>$ have been proved in \cite{DAdderio-VandenWyngaerd-2017} and \cite{DAdderio-Iraci-polyominoes-2017} respectively. Moreover, combined with the results in \cite{DAdderio-VandenWyngaerd-2017}, a ``compositional'' refinement of the case $\<\cdot , e_{n-d}h_d\>$ is proved in \cite{Zabrocki-4Catalan-2016}.

Again in \cite{Haglund-Remmel-Wilson-2015}, the authors proposed a generalization of the Delta conjecture, which predicts a combinatorial interpretation of $\Delta_{h_m}\Delta_{e_{n-k-1}}'e_n$ in terms of decorated partially labelled Dyck paths: this is what we call the \emph{generalized Delta conjecture}.

In the present work we prove the so called Schr\"{o}der case of the generalized Delta conjecture, i.e. the special case $\<\cdot , e_{n-d}h_d\>$. This extends some of the results in \cite{DAdderio-VandenWyngaerd-2017}, which in turn extended the $q,t$-Schr\"{o}der of Haglund \cite{Haglund-Schroeder-2004}.

As in the aforementioned works, to prove our result, we will give a recursion, which generalizes the ones in the literature. The essential tool on the symmetric function side will be the most technical result of \cite{DAdderio-VandenWyngaerd-2017}.

Also, we introduce a statistic \emph{bounce} which will provide another combinatorial interpretation of the same polynomial, and we give a statistic preserving bijection between the two interpretations.

It turns out that our recursion is given by iterating a simpler recursion. We will explain this intermediate recursion by giving another combinatorial interpretation of the same polynomial in terms of doubly decorated parallelogram polyominoes. This result extends some of the results in \cite{DAdderio-Iraci-polyominoes-2017}, which in turn extended the results in \cite{Aval-DAdderio-Dukes-Hicks-LeBorgne-2014}.

\medskip

The paper is organized in the following way: in Section~2 we recall the statement of the generalized Delta conjecture of \cite{Haglund-Remmel-Wilson-2015} by introducing the relevant definitions and notations. In Section~3 we establish the needed results of symmetric function theory. In particular we introduce a family of plethystic formulae, and we show that they satisfy a certain recursion and that they sum to the case $\<\cdot ,e_{n-d}h_d\>$ of the generalized Delta conjecture. In Section~4 we prove the main result of this work, i.e. the case $\<\cdot ,e_{n-d}h_d\>$ of the generalized Delta conjecture. Moreover, we prove also another combinatorial interpretation in terms of a new bounce statistic. We conclude the section with a statistic preserving bijection. In Section~5 we prove two further combinatorial interpretations of the same polynomial in terms of doubly decorated parallelogram polyominoes, and we give a bijection explaining their equivalence. In Section~6 we explain combinatorially the relation between the results in Section~4 and the ones in Section~5. Finally, in Section~7 we state an intriguing open problem.

\section{Statement of the generalized Delta conjecture}

In \cite{Haglund-Remmel-Wilson-2015}, the authors conjectured a combinatorial interpretation for the symmetric function \[ \Delta_{h_m}\Delta'_{e_{n-k-1}}e_n \] in terms of partially labelled decorated Dyck paths, known as the \emph{generalized Delta conjecture} because it reduces to the Delta conjecture when $m=0$. We give the necessary definitions.

\begin{definition}
	A \emph{Dyck path} of size $n$ is a lattice path going from $(0,0)$ to $(n,n)$, using only north and east steps and staying weakly above the line $x=y$ (also called the \emph{main diagonal}). The set of Dyck paths of size $n$ will be denoted by $\mathsf{D}(n)$. A \emph{partially labelled Dyck path} is a Dyck path whose vertical steps are labelled with (not necessarily distinct) non-negative integers such that the labels appearing in each column are strictly increasing from bottom to top, and $0$ does not appear in the first column. The set of partially labelled Dyck paths with $m$ zero labels and $n$ nonzero labels is denoted by $\PLD(m,n)$.  
\end{definition}

Partially labelled Dyck paths differ from labelled Dyck paths only in that $0$ is allowed as a label in the former and not in the latter. 

\begin{definition}
	We define for each $D\in \PLD(m,n)$ a monomial in the variables $x_1,x_2,\dots$ we set \[ x^D \coloneqq \prod_{i=1}^n x_{l_i(D)} \] where $l_i(D)$ is the label of the $i$-th vertical step of $D$ (the first being at the bottom). Notice that $x_0$ does not appear, which explains the word \emph{partially}.
\end{definition}

\begin{definition}
	Let $D$ be a (partially labelled) Dyck path of size $n+m$. We define its \emph{area word} to be the string of integers $a(D) = a_1(D) \cdots a_{n+m}(D)$ where $a_i(D)$ is the number of whole squares in the $i$-th row (from the bottom) between the path and the main diagonal.
\end{definition}

\begin{definition} \label{def: rise}
	The \emph{rises} of a Dyck path $D$ are the indices \[ \Rise(D) \coloneqq \{2\leq i \leq n+m\mid a_{i}(D)>a_{i-1}(D)\},\] or the vertical steps that are directly preceded by another vertical step. Taking a subset $\DRise(D)\subseteq \Rise (D)$ and decorating the corresponding vertical steps with a $\ast$, we obtain a \emph{decorated Dyck path}, and we will refer to these vertical steps as \emph{decorated rises}. 
\end{definition}

The set of partially labelled decorated Dyck paths with $m$ zero labels, $n$ nonzero labels and $k$ decorated rises is denoted by $\PLD(m,n)^{\ast k}$. See Figure~\ref{fig:pldExample1} for an example. 

We define two statistics on this set.

\begin{definition} \label{def: area DP}
	We define the \emph{area} of a (partially labelled) decorated Dyck path $D$ as \[ \area(D) \coloneqq \sum_{i\not \in \DRise(D)} a_i(D). \]
\end{definition}

For a more visual definition, the area is the number of whole squares that lie between the path and the main diagonal, except for the ones in the rows containing a decorated rise. For example, the decorated Dyck path in Figure~\ref{fig:pldExample1} has area $6$. 

Notice that the area does not depend on the labels. 

\begin{definition} \label{def: dinv DP}
	Let $D \in \PLD(m,n)$. For $1 \leq i < j \leq n+m$, we say that the pair $(i,j)$ is an \emph{inversion} if
	\begin{itemize}
		\item either $a_i(D) = a_j(D)$ and $l_i(D) < l_j(D)$ (\emph{primary inversion}),
		\item or $a_i(D) = a_j(D) + 1$ and $l_i(D) > l_j(D)$ (\emph{secondary inversion}).
	\end{itemize}
	Then we define \[\dinv(D)\coloneqq \# \{0 \leq i < j \leq n+m \mid (i,j) \; \text{is an inversion}. \}\]
\end{definition}

For example, the decorated Dyck path in Figure~\ref{fig:pldExample1} has $1$ primary inversion (the pair $(2,4)$) and $2$ secondary inversions (the pairs $(2,3)$ and $(5,6)$), so its dinv is $3$. Notice that the decorations on the rises do not affect the dinv.

The following conjecture is stated in \cite{Haglund-Remmel-Wilson-2015}.

\begin{conjecture}[Generalized Delta]
	\[ \Delta_{h_{m}} \Delta'_{e_{n-k-1}} e_{n} = \sum_{D \in \PLD(m,n)^{\ast k}} q^{\dinv(D)} t^{\area(D)} x^D. \]
\end{conjecture}

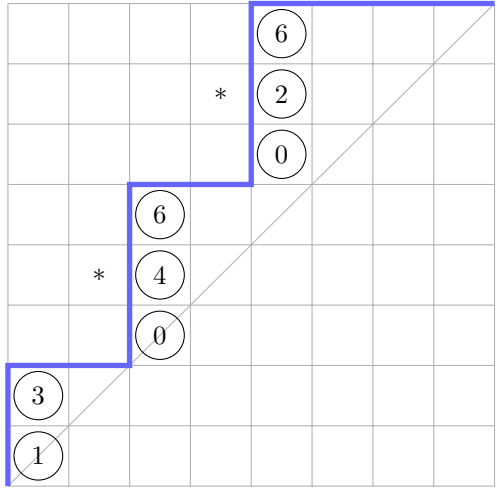
\begin{figure*}[!ht]
	\centering
	\begin{tikzpicture}[scale = .8]
	
	\draw[step=1.0, gray!60, thin] (0,0) grid (8,8);
	
	\draw[gray!60, thin] (0,0) -- (8,8);
	
	\draw[blue!60, line width=2pt] (0,0) -- (0,1) -- (0,2) -- (1,2) -- (2,2) -- (2,3) -- (2,4) -- (2,5) -- (3,5) -- (4,5) -- (4,6) -- (4,7) -- (4,8) -- (5,8) -- (6,8) -- (7,8) -- (8,8);
	
	\draw (0.5,0.5) circle (0.4 cm) node {$1$};
	\draw (0.5,1.5) circle (0.4 cm) node {$3$};
	\draw (2.5,2.5) circle (0.4 cm) node {$0$};
	\draw (2.5,3.5) circle (0.4 cm) node {$4$};
	\draw (2.5,4.5) circle (0.4 cm) node {$6$};
	\draw (4.5,5.5) circle (0.4 cm) node {$0$};
	\draw (4.5,6.5) circle (0.4 cm) node {$2$};
	\draw (4.5,7.5) circle (0.4 cm) node {$6$};
	
	\node at (1.5,3.5) {$\ast$};
	\node at (3.5,6.5) {$\ast$};
	
	\end{tikzpicture}
	\caption{Example of an element in $\PLD(2,6)^{\ast 2}$.}
	\label{fig:pldExample1}
\end{figure*}

\section{Symmetric functions}

For all the undefined notations and the unproven identities, we refer to \cite{DAdderio-VandenWyngaerd-2017}*{Section~1}, where definitions, proofs, and/or references can be found. In the next subsection we will limit ourselves to introduce some notation, while in the following one we will recall some identities that are going to be useful in the sequel. In the third and final subsection we will prove the main results on symmetric functions of this work.

For more references on symmetric functions cf. also \cite{Macdonald-Book-1995}, \cite{Stanley-Book-1999} and \cite{Haglund-Book-2008}.

\subsection{Notation}

We denote by $\Lambda=\bigoplus_{n\geq 0}\Lambda^{(k)}$ the graded algebras of symmetric functions with coefficients in $\mathbb{Q}(q,t)$, and by $\<\, , \>$ the \emph{Hall scalar product} on $\Lambda$, which can be defined by saying that the Schur functions form an orthonormal basis.

The standard bases of the symmetric functions that will appear in our
calculations are the complete $\{h_{\lambda}\}_{\lambda}$, elementary $\{e_{\lambda}\}_{\lambda}$, power $\{p_{\lambda}\}_{\lambda}$ and Schur $\{s_{\lambda}\}_{\lambda}$ bases.

\emph{We will use implicitly the usual convention that $e_0 = h_0 = 1$ and $e_k = h_k = 0$ for $k < 0$.}

For a partition $\mu\vdash n$, we denote by
\begin{align}
	\widetilde{H}_{\mu} \coloneqq \widetilde{H}_{\mu}[X]=\widetilde{H}_{\mu}[X;q,t]=\sum_{\lambda\vdash n}\widetilde{K}_{\lambda \mu}(q,t)s_{\lambda}
\end{align}
the \emph{(modified) Macdonald polynomials}, where
\begin{align}
	\widetilde{K}_{\lambda \mu} \coloneqq \widetilde{K}_{\lambda \mu}(q,t)=K_{\lambda \mu}(q,1/t)t^{n(\mu)}\quad \text{ with }\quad n(\mu)=\sum_{i\geq 1}\mu_i(i-1)
\end{align}
are the \emph{(modified) Kostka coefficients} (see \cite{Haglund-Book-2008}*{Chapter~2} for more details). 

The set $\{\widetilde{H}_{\mu}[X;q,t]\}_{\mu}$ is a basis of the ring of symmetric functions $\Lambda$. This is a modification of the basis introduced by Macdonald \cite{Macdonald-Book-1995}.

If we identify the partition $\mu$ with its Ferrers diagram, i.e. with the collection of cells $\{(i,j)\mid 1\leq i\leq \mu_i, 1\leq j\leq \ell(\mu)\}$, then for each cell $c\in \mu$ we refer to the \emph{arm}, \emph{leg}, \emph{co-arm}, and \emph{co-leg} (denoted respectively as $a_\mu(c), l_\mu(c), a_\mu(c)', l_\mu(c)'$) as the number of cells in $\mu$ that are strictly to the right, above, to the left and below $c$ in $\mu$, respectively.

We set $M \coloneqq (1-q)(1-t)$ and we define for every partition $\mu$
\begin{align}
	B_{\mu} & \coloneqq B_{\mu}(q,t)=\sum_{c\in \mu}q^{a_{\mu}'(c)}t^{l_{\mu}'(c)} \\
	T_{\mu} & \coloneqq T_{\mu}(q,t)=\prod_{c\in \mu}q^{a_{\mu}'(c)}t^{l_{\mu}'(c)} \\
	\Pi_{\mu} & \coloneqq \Pi_{\mu}(q,t)=\prod_{c\in \mu/(1)}(1-q^{a_{\mu}'(c)}t^{l_{\mu}'(c)}) \\
	w_{\mu} & \coloneqq w_{\mu}(q,t)=\prod_{c\in \mu} (q^{a_{\mu}(c)} - t^{l_{\mu}(c) + 1}) (t^{l_{\mu}(c)} - q^{a_{\mu}(c) + 1}).
\end{align}

We will make extensive use of the \emph{plethystic notation} (cf. \cite{Haglund-Book-2008}*{Chapter~1}).

We define the \emph{nabla} operator on $\Lambda$ by
\begin{align}
	\nabla \widetilde{H}_{\mu} \coloneqq T_{\mu} \widetilde{H}_{\mu} \quad \text{ for all } \mu,
\end{align}
and we define the \emph{delta} operators $\Delta_f$ and $\Delta_f'$ on $\Lambda$ by
\begin{align}
	\Delta_f \widetilde{H}_{\mu} \coloneqq f[B_{\mu}(q,t)] \widetilde{H}_{\mu} \quad \text{ and } \quad 
	\Delta_f' \widetilde{H}_{\mu}  \coloneqq f[B_{\mu}(q,t)-1] \widetilde{H}_{\mu}, \quad \text{ for all } \mu.
\end{align}
Observe that on the vector space of symmetric functions homogeneous of degree $n$, denoted by $\Lambda^{(n)}$, the operator $\nabla$ equals $\Delta_{e_n}$. Moreover, for every $1\leq k\leq n$,
\begin{align}
	\label{eq:deltaprime}
	\Delta_{e_k} = \Delta_{e_k}' + \Delta_{e_{k-1}}' \quad \text{ on } \Lambda^{(n)},
\end{align}
and for any $k > n$, $\Delta_{e_k} = \Delta_{e_{k-1}}' = 0$ on $\Lambda^{(n)}$, so that $\Delta_{e_n}=\Delta_{e_{n-1}}'$ on $\Lambda^{(n)}$.

\medskip

For a given $k\geq 1$, we define the \emph{Pieri coefficients} $c_{\mu \nu}^{(k)}$ and $d_{\mu \nu}^{(k)}$ by setting
\begin{align}
	\label{eq:def_cmunu} h_{k}^\perp \widetilde{H}_{\mu}[X] & =\sum_{\nu \subset_k \mu} c_{\mu \nu}^{(k)} \widetilde{H}_{\nu}[X], \\
	\label{eq:def_dmunu} e_{k}\left[\frac{X}{M}\right] \widetilde{H}_{\nu}[X] & = \sum_{\mu \supset_k \nu} d_{\mu \nu}^{(k)} \widetilde{H}_{\mu}[X],
\end{align}
where $\nu\subset_k \mu$ means that $\nu$ is contained in $\mu$ (as Ferrers diagrams) and $\mu/\nu$ has $k$ lattice cells, and the symbol $\mu \supset_k \nu$ is analogously defined. The following identity is well-known:
\begin{align}
	\label{eq:rel_cmunu_dmunu}
	c_{\mu \nu}^{(k)} = \frac{w_{\mu}}{w_{\nu}}d_{\mu \nu}^{(k)}.
\end{align}

\medskip

We will also use the symmetric functions $E_{n,k}$, that were introduced in \cite{Garsia-Haglund-qtCatalan-2002} by means of the following expansion:
\begin{align}
	\label{eq:def_Enk}
	e_n \left[ X \frac{1-z}{1-q} \right] = \sum_{k=1}^n \frac{(z;q)_k}{(q;q)_k} E_{n,k},
\end{align}
where
\begin{align}
	(a;q)_s \coloneqq (1 - a)(1 - qa)(1 - q^2 a) \cdots (1 - q^{s-1} a)
\end{align}
is the usual $q$-\emph{rising factorial}.

Observe that 
\begin{equation} \label{eq:Enk}
e_n=\sum_{k=1}^nE_{n,k}.
\end{equation}

Recall also the standard notation for $q$-analogues: for $n, k\in \mathbb{N}$, we set
\begin{align}
	[0]_q \coloneqq 0, \quad \text{ and } \quad [n]_q \coloneqq \frac{1-q^n}{1-q} = 1+q+q^2+\cdots+q^{n-1} \quad \text{ for } n \geq 1,
\end{align}
\begin{align}
	[0]_q! \coloneqq 1 \quad \text{ and }\quad [n]_q! \coloneqq [n]_q[n-1]_q \cdots [2]_q [1]_q \quad \text{ for } n \geq 1,
\end{align}
and
\begin{align}
	\qbinom{n}{k}_q  \coloneqq \frac{[n]_q!}{[k]_q![n-k]_q!} \quad \text{ for } n \geq k \geq 0, \quad \text{ while } \quad \qbinom{n}{k}_q \coloneqq 0 \quad \text{ for } n < k.
\end{align}

\subsection{Some basic identities}

The following identity is well-known: for any symmetric function $f\in \Lambda^{(n)}$,
\begin{align}
	\label{eq:lem_e_h_Delta}
	\< \Delta_{e_{d}} f, h_n \> = \< f, e_d h_{n-d} \>.
\end{align}

We will use the following form of \emph{Macdonald-Koornwinder reciprocity}: for all partitions $\alpha$ and $\beta$
\begin{align}
	\label{eq:Macdonald_reciprocity}
	\frac{\widetilde{H}_{\alpha}[MB_{\beta}]}{\Pi_{\alpha}} = \frac{\widetilde{H}_{\beta}[MB_{\alpha}]}{\Pi_{\beta}}.
\end{align}
The following identity is also known as \emph{Cauchy identity}:
\begin{align}
	\label{eq:Mac_Cauchy}
	e_n \left[ \frac{XY}{M} \right] = \sum_{\mu \vdash n} \frac{ \widetilde{H}_{\mu} [X] \widetilde{H}_\mu [Y]}{w_\mu} \quad \text{ for all } n.
\end{align}

We need the following well known proposition.
\begin{proposition} 
	For $n\in \mathbb{N}$ we have
	\begin{align}
		\label{eq:en_expansion}
		e_n[X] = e_n \left[ \frac{XM}{M} \right] = \sum_{\mu \vdash n} \frac{M B_\mu \Pi_{\mu} \widetilde{H}_\mu[X]}{w_\mu}.
	\end{align}
	Moreover, for all $k\in \mathbb{N}$ with $0\leq k\leq n$, we have
	\begin{align}
		\label{eq:e_h_expansion}
		h_k \left[ \frac{X}{M} \right] e_{n-k} \left[ \frac{X}{M} \right] = \sum_{\mu \vdash n} \frac{e_k[B_\mu] \widetilde{H}_\mu[X]}{w_\mu}.
	\end{align}
\end{proposition}

Using \eqref{eq:Mac_Cauchy} with $Y = [j]_q = \frac{1-q^j}{1-q}$, we get the following well-known expansion:
\begin{align} \label{eq:qn_q_Macexp}
	e_n \left[ X \frac{1-q^j}{1-q} \right] & = \sum_{\mu\vdash n} \frac{\widetilde{H}_\mu[X] \widetilde{H}_\mu[M[j]_q]}{w_\mu}\\
	& = (1-q^j) \sum_{\mu\vdash n} \frac{ \Pi_\mu \widetilde{H}_\mu[X] h_j[(1-t)B_\mu]}{w_\mu}.
\end{align}

We need another theorem of Haglund: the following is essentially \cite{Haglund-Schroeder-2004}*{Theorem~2.5}.
\begin{theorem}
	For $k,n\in \mathbb{N}$ with $1\leq k\leq n$, 
	\begin{align}
		\label{eq:Haglund_nablaEnk}
		\nabla E_{n,k} = t^{n-k}(1-q^k) \mathbf{\Pi} h_k \left[ \frac{X}{1-q} \right] h_{n-k} \left[ \frac{X}{M} \right],
	\end{align}
	where $\mathbf{\Pi}$ is the invertible linear operator defined by
	\begin{equation} \label{eq:defPi}
	\mathbf{\Pi} \widetilde{H}_\mu[X] = \Pi_\mu \widetilde{H}_\mu[X] \qquad \text{ for all } \mu.
	\end{equation}
\end{theorem}

The main ingredient for proving our recursion on the symmetric function side is the following crucial theorem from \cite{DAdderio-VandenWyngaerd-2017}.
\begin{theorem}[\cite{DAdderio-VandenWyngaerd-2017}*{Theorem~3.1}]
	For $m,k\geq 1$ and $\ell\geq 0$, we have
	\begin{align} \label{eq:mastereq}
		\sum_{\gamma\vdash m}\frac{\widetilde{H}_\gamma[X]}{w_\gamma} h_k[(1-t)B_\gamma]e_\ell[B_\gamma]  & = \sum_{j=0}^{\ell} t^{\ell-j}\sum_{s=0}^{k}q^{\binom{s}{2}}\qbinom{s+j}{s}_q \qbinom{k+j-1}{s+j-1}_q \\
		\notag & \times  h_{s+j}\left[\frac{X}{1-q}\right] h_{\ell-j}\left[\frac{X}{M}\right] e_{m-s-\ell}\left[\frac{X}{M}\right].
	\end{align}
\end{theorem}

\subsection{The family $F_{n,k;p}^{(d,\ell)}$}

Set
\begin{equation}
F_{n,k;p}^{(d,\ell)} \coloneqq t^{n-k-\ell} \< \Delta_{h_{n-k-\ell}} \Delta_{e_\ell} e_{n+p-d} \left[ X \frac{1-q^k}{1-q} \right], e_p h_{n-d} \>.
\end{equation}
Notice that $F_{n,k;0}^{(d,\ell)}=F_{n,k}^{(d,\ell)}$ of \cite{DAdderio-VandenWyngaerd-2017}*{Section~4}.

The family of plethystic formulae $F_{n,k;p}^{(d,\ell)}$ satisfy the following recursion.
\begin{theorem}
	For $k,\ell,d,p\geq 0$, $n\geq k+\ell$ and $n+p\geq d$, the $F_{n,k;p}^{(d,\ell)}$ satisfy the following recursion: for $n\geq 1$ 
	\begin{equation} \label{eq:reconn}
	F_{n,n;p}^{(d,\ell)}=\delta_{\ell,0}q^{\binom{n-d}{2}}\qbinom{n}{n-d}\qbinom{n+p-1}{p}
	\end{equation}
	and, for $n\geq 1$ and $1\leq k<n$,
	\begin{equation} \label{eq:recogen}
	F_{n,k;p}^{(d,\ell)} = t^{n-\ell-k} \sum_{j=0}^{p} \sum_{s=0}^{k} q^{\binom{s}{2}} \qbinom{k}{s}_q \qbinom{k+j-1}{j}_q F_{n+p-d,s+j;n-\ell-k}^{(n+p-d-\ell,n-d-s)},
	\end{equation}
	with initial conditions
	\begin{equation}
	F_{0,k;p}^{(d,\ell)}=\delta_{k,0}\delta_{p,0}\delta_{d,0}\delta_{\ell,0}\qquad \text{and} \qquad F_{n,0;p}^{(d,\ell)}=\delta_{n,0}\delta_{p,0}\delta_{d,0}\delta_{\ell,0}.
	\end{equation}
\end{theorem}
\begin{proof}
	Because of the operator $\Delta_{h_{n-k-\ell}}$ in the definition of $F_{n,k;p}^{(d,\ell)}$, it is clear that $F_{n,n;p}^{(d,\ell)}=0$ for $\ell\neq 0$. For $\ell=0$, using \eqref{eq:qn_q_Macexp}, we get
	\begin{align*}
		F_{n,n;p}^{(d,0)} &  = \left\langle e_{n+p-d}\left[X\frac{1-q^n}{1-q}\right],h_{n-d}e_p\right\rangle\\
		& = \sum_{\gamma\vdash n+p-d}\frac{\widetilde{H}_\gamma[M[n]_q]}{w_\gamma}e_p[B_\gamma]\\
		\text{(using \eqref{eq:e_h_expansion})}& = h_p[[n]_q] e_{n-d}[[n]_q]\\
		& = \qbinom{n+p-1}{p}_q q^{\binom{n-d}{2}}\qbinom{n}{n-d}_q
	\end{align*}
	where in the last equality we used well-known identities (cf. \cite{Stanley-Book-1999}*{Theorem~7.21.2}). This proves \eqref{eq:reconn}. For $k<n$ we have
	\begin{align*}
		F_{n,k;p}^{(d,\ell)} & =t^{n-\ell-k}\< \Delta_{h_{n-\ell-k}}\Delta_{e_{\ell}} e_{n+p-d}\left[X\frac{1-q^k}{1-q}\right] ,e_ph_{n-d}\> \\
		\text{(using \eqref{eq:qn_q_Macexp}, \eqref{eq:lem_e_h_Delta})}	& = t^{n-\ell-k}(1-q^k)\sum_{\gamma\vdash n+p-d}\frac{\Pi_\gamma}{w_\gamma}h_k[(1-t)B_\gamma]h_{n-\ell-k}[B_\gamma]e_\ell[B_\gamma] e_p[B_\gamma]\\
		\text{(using \eqref{eq:e_h_expansion})}& = t^{n-\ell-k}(1-q^k)\sum_{\gamma\vdash n+p-d}\frac{\Pi_\gamma}{w_\gamma}h_k[(1-t)B_\gamma]e_p[B_\gamma]\sum_{\mu\vdash n-k}e_{n-\ell-k}[B_\mu]\frac{\widetilde{H}_\mu[MB_\gamma]}{w_\mu} \\
		\text{(using \eqref{eq:Macdonald_reciprocity})}& = t^{n-\ell-k}\sum_{\mu\vdash n-k}  e_{n-\ell-k}[B_\mu]\frac{\Pi_\mu}{w_\mu} (1-q^k)\sum_{\gamma\vdash n+p-d}\frac{\widetilde{H}_\gamma[MB_\mu]}{w_\gamma} h_k[(1-t)B_\gamma]e_p[B_\gamma]\\
		\text{(using \eqref{eq:mastereq})}& = t^{n-\ell-k}\sum_{\mu\vdash n-\ell-k+p}  e_{n-\ell-k}[B_\mu]\frac{\Pi_\mu}{w_\mu} (1-q^k)\sum_{j=0}^{p} t^{p-j}\sum_{s=0}^{k}q^{\binom{s}{2}}\qbinom{s+j}{s}_q\\
		& \quad \times  \qbinom{k+j-1}{s+j-1}_q h_{s+j}\left[(1-t)B_\mu \right] h_{p-j}\left[MB_\mu \right] e_{n-d-s}\left[MB_\mu \right]\\
		& = t^{n-\ell-k}\sum_{j=0}^{p} t^{p-j}\sum_{s=0}^{k}q^{\binom{s}{2}}\qbinom{k}{s}_q	\qbinom{k+j-1}{j}_q	(1-q^{s+j}) \\
		& \quad \times \sum_{\mu\vdash n-k}  e_{n-\ell-k}[B_\mu]\frac{\Pi_\mu}{w_\mu} h_{s+j}\left[(1-t)B_\mu \right] h_{p-j}\left[MB_\mu \right] e_{n-d-s}\left[MB_\mu \right] \\
		\text{(using \eqref{eq:qn_q_Macexp})}	& = t^{n-\ell-k}\sum_{j=0}^{p} \sum_{s=0}^{k}q^{\binom{s}{2}}\qbinom{k}{s}_q	\qbinom{k+j-1}{j}_q	 \\
		& \quad \times t^{p-j}\<\Delta_{h_{p-j}} \Delta_{e_{n-d-s}}e_{n-k}\left[X\frac{1-q^{s+j}}{1-q}\right],e_{n-\ell-k} h_{\ell}\>\\
		& = t^{n-\ell-k}\sum_{j=0}^{p} \sum_{s=0}^{k}q^{\binom{s}{2}}\qbinom{k}{s}_q	\qbinom{k+j-1}{j}_q F_{n+p-d,s+j;n-\ell-k}^{(n+p-d-\ell,n-d-s)}.
	\end{align*}
	This proves \eqref{eq:recogen}. The initial conditions are easy to check.
\end{proof}

Iterating the previous recursion we get the following immediate corollary.

\begin{corollary}\label{cor: iterated recursion}
	For $k,\ell,d,p\geq 0$, $n\geq k+\ell$ and $n+p\geq d$, the $F_{n,k;p}^{(d,\ell)}$ satisfy the following recursion: for $n\geq 1$ 
	\begin{equation*}
		F_{n,n;p}^{(d,\ell)}=\delta_{\ell,0}q^{\binom{n-d}{2}}\qbinom{n}{n-d}\qbinom{n+p-1}{p}
	\end{equation*}
	and, for $n\geq 1$ and $1\leq k<n$,
	\begin{align*}
		F_{n,k;p}^{(d,\ell)} & = t^{n-k-\ell} \sum_{j=0}^{p} \sum_{s=0}^{k} q^{\binom{s}{2}} \qbinom{k}{s}_q \qbinom{k+j-1}{j}_q \\ 
		& \times  t^{p-j} \sum_{u=0}^{n-k-\ell} \sum_{v=0}^{s+j} q^{\binom{v}{2}} \qbinom{s+j}{v}_q \qbinom{s+j+u-1}{u}_q F_{n-k,u+v;p-j}^{(d-k+s,\ell-v)},
	\end{align*}
	with initial conditions
	\begin{equation}
	F_{0,k;p}^{(d,\ell)}=\delta_{k,0}\delta_{p,0}\delta_{d,0}\delta_{\ell,0} \qquad \text{and} \qquad F_{n,0;p}^{(d,\ell)}=\delta_{n,0}\delta_{p,0}\delta_{d,0}\delta_{\ell,0}.
	\end{equation}
\end{corollary}

We need a lemma.
\begin{lemma}
	For $k,\ell,d,p\geq 0$, $n\geq k+\ell$ and $n+p\geq d$, we have
	\begin{equation} \label{eq:lemnablaEnk}
	F_{n,k;p}^{(d,\ell)} = \sum_{\gamma\vdash n+p-d}\left.(\mathbf{\Pi}^{-1}\nabla E_{n-\ell,k}[X])\right|_{X=MB_\gamma} \frac{\Pi_\gamma}{w_\gamma}e_{\ell}[B_\gamma]e_p[B_\gamma],
	\end{equation}
	where $\mathbf{\Pi}$ is the operator defined in \eqref{eq:defPi}.
\end{lemma}
\begin{proof}
	We have
	\begin{align*}
		\sum_{k=1}^{n-\ell}F_{n,k;p}^{(d,\ell)}	& =  \sum_{k=1}^{n-\ell}t^{n-\ell-k}\left\langle \Delta_{h_{n-\ell-k}}\Delta_{e_{\ell}}e_{n+p-d}\left[X\frac{1-q^k}{1-q}\right],h_{n-d}e_p\right\rangle\\
		& = \sum_{k=1}^{n-\ell}t^{n-\ell-k}(1-q^k)\sum_{\gamma\vdash n+p-d}\frac{\Pi_\gamma}{w_\gamma}h_k[(1-t)B_\gamma] h_{n-\ell-k}[B_\gamma]e_{\ell}[B_\gamma]e_p[B_\gamma]\\
		& = \sum_{k=1}^{n-\ell} \sum_{\gamma\vdash n+p-d}\left.(\mathbf{\Pi}^{-1}\nabla E_{n-\ell,k}[X])\right|_{X=MB_\gamma} \frac{\Pi_\gamma}{w_\gamma}e_{\ell}[B_\gamma]e_p[B_\gamma],
	\end{align*}
	where in the second equality we used \eqref{eq:qn_q_Macexp} and \eqref{eq:lem_e_h_Delta}, and in the last one we used \eqref{eq:Haglund_nablaEnk}.
\end{proof}

The interest in the polynomials $F_{n,k;p}^{(d,\ell)}$ lies in the following theorem.

\begin{theorem} \label{thm:recoFnkpdl}
	For $\ell,d,p\geq 0$, $n\geq \ell+1$ and $n\geq d$, we have
	$$
	\sum_{k=1}^{n-\ell}F_{n,k;p}^{(d,\ell)}= \< \Delta_{h_p}\Delta_{e_{n-\ell-1}}'e_n, e_{n-d}h_{d} \>.
	$$
\end{theorem}
\begin{proof}
	Using \eqref{eq:lemnablaEnk}, we have
	\begin{align*}
		\sum_{k=1}^{n-\ell}F_{n,k;p}^{(d,\ell)}	& = \sum_{k=1}^{n-\ell} \sum_{\gamma\vdash n+p-d}\left.(\mathbf{\Pi}^{-1}\nabla E_{n-\ell,k}[X])\right|_{X=MB_\gamma} \frac{\Pi_\gamma}{w_\gamma}e_{\ell}[B_\gamma]e_p[B_\gamma]\\
		\text{(using \eqref{eq:def_Enk})}& =  \sum_{\gamma\vdash n+p-d}\left.(\mathbf{\Pi}^{-1}\nabla e_{n-\ell}[X])\right|_{X=MB_\gamma} \frac{\Pi_\gamma}{w_\gamma}e_{\ell}[B_\gamma]e_p[B_\gamma]\\
		\text{(using \eqref{eq:en_expansion})} & =  \sum_{\gamma\vdash n+p-d}\sum_{\mu\vdash n-\ell} M B_\mu T_\mu \frac{\widetilde{H}_\mu[MB_\gamma]}{w_\mu} \frac{\Pi_\gamma}{w_\gamma}e_{\ell}[B_\gamma]e_p[B_\gamma]\\
		& =  \sum_{\gamma\vdash n+p-d}\sum_{\mu\vdash n-\ell} \frac{M B_\mu T_\mu }{w_\mu} \frac{\Pi_\gamma}{w_\gamma}e_p[B_\gamma] e_{\ell}\left[\frac{MB_\gamma}{M}\right]\widetilde{H}_\mu[MB_\gamma]\\
		& =  \sum_{\gamma\vdash n+p-d}\sum_{\mu\vdash n-\ell} \frac{M B_\mu T_\mu }{w_\mu} \frac{\Pi_\gamma}{w_\gamma}e_p[B_\gamma] \sum_{\alpha\supset_\ell \mu}d_{\alpha \mu}^{(\ell)} \widetilde{H}_\alpha[MB_\gamma]\\
		\text{(using \eqref{eq:rel_cmunu_dmunu})}& =  \sum_{\alpha\vdash n}\frac{M\Pi_\alpha}{w_\alpha}\sum_{\mu\subset_\ell \alpha}    B_\mu T_\mu   c_{\alpha \mu}^{(\ell)} \sum_{\gamma\vdash n+p-d}e_p[B_\gamma]\frac{\widetilde{H}_\gamma[MB_\alpha]}{w_\gamma}\\
		\text{(using \eqref{eq:e_h_expansion})} & =  \sum_{\alpha\vdash n}\frac{M\Pi_\alpha}{w_\alpha}h_p[B_\alpha]e_{n-d}[B_\alpha]\sum_{\mu\subset_\ell \alpha}    B_\mu T_\mu   c_{\alpha \mu}^{(\ell)} \\
		& =  \sum_{\alpha\vdash n}\frac{M\Pi_\alpha}{w_\alpha}h_p[B_\alpha]e_{n-d}[B_\alpha] e_{n-\ell-1}[B_\alpha-1] B_\alpha \\
		\text{(using \eqref{eq:lem_e_h_Delta})}& = \< \Delta_{h_p}\Delta_{e_{n-\ell-1}}'e_n, e_{n-d}h_{d} \>.
	\end{align*}
	where in the second to last equality we used \cite{DAdderio-VandenWyngaerd-2017}*{Lemma~5.2}.
\end{proof}

\section{Decorated Dyck paths}

\begin{definition}
	The \emph{valleys} of a Dyck path $D\in \mathsf{D}(n)$ are the indices
	\begin{align*}
		\Val(D) \coloneqq & \; \{2\leq i\leq n \mid a_i(D)<a_{i-1}(D)\},
	\end{align*}
	or the vertical steps that are directly preceded by a horizontal step.
\end{definition}

\begin{definition}
	The \emph{peaks} of a Dyck path $D\in \mathsf{D}(n)$ are the indices \[ \Peak(D) \coloneqq \{1\leq i\leq n-1 \mid a_{i+1}(D) \leq a_i(D)\}\cup \{n\}, \] or the vertical steps that are followed by a horizontal step.
\end{definition}

\begin{definition}
	Fix $n,m,a,b\in \mathbb{N}$, $n\geq 1$, $m,a,b\geq 0$. For every Dyck path $D\in \mathsf{D}(n+m)$ with $|\Rise(D)|\geq a$, $|\Peak(D)|\geq b$ and $|\Val(D)|\geq m$ choose three subsets of $\{1,\dots, m+n\}$:
	\begin{enumerate}[(i)]
		\item $\DRise(D)\subseteq \Rise(D)$ (see Definition \ref{def: rise}) such that $\vert \DRise(D) \vert = a$ and decorate the corresponding vertical steps with a $\ast$. 
		\item $\DPeak(D)\subseteq \Peak(D)$ such that $\vert \DPeak(D) \vert = b$ and decorate with a $\textcolor{blue}{\bullet}$ the points joining these vertical steps with the horizontal steps following them. We will call these \emph{decorated peaks}. 
		\item $\ZVal(D)\subseteq \Val(D)$ such that $\vert \ZVal \vert = m$ and $\DPeak(D)\cap \ZVal(D)= \emptyset$. Label the corresponding vertical steps with a zero. These steps will be called \emph{zero valleys}.
	\end{enumerate} 
	We denote the set of these paths by $\DD(m,n)^{\ast a, \circ b}$. See Figure~\ref{fig:pldExample2} for an example. 
\end{definition}

\begin{figure}[ht]	
	\begin{minipage}{.5 \textwidth}
		\begin{center}
			\begin{tikzpicture}[scale=.7]
			
			\draw[step=1.0, gray!60, thin] (0,0) grid (8,8);
			
			\draw[gray!60, thin] (0,0) -- (8,8);
			
			\draw[blue!60, line width=1.6pt] (0,0) -- (0,1) -- (0,2) -- (1,2) -- (1,3) -- (2,3) -- (3,3) -- (3,4) -- (3,5) -- (4,5) -- (4,6) -- (5,6) -- (5,7) -- (5,8) -- (8,8);
			
			\filldraw[fill=blue!60]
			(1,3) circle (4 pt)
			(5,8) circle (4 pt);
			
			\node at (3.5,3.5) {$0$};
			\draw (3.5, 3.5) circle (.4cm);
			\node at (5.5,6.5) {$0$};
			\draw (5.5, 6.5) circle (.4cm);
			\node at (1-1-0.5,1+0.5) {$\ast$};
			\node at (5-2-0.5,4+0.5) {$\ast$};
			\end{tikzpicture}
			\caption{Example of a decorated Dyck path in $\DD(2,4)^{\ast 2, \circ 2}$. }	
			\label{fig:pldExample2}
		\end{center}
	\end{minipage}%
	\begin{minipage}{.5\textwidth}
		\begin{center}
			\ 	\begin{tikzpicture}[scale=.7]
			
			\draw[step=1.0, gray!60, thin] (0,0) grid (8,8);
			
			\draw[gray!60, thin] (0,0) -- (8,8);
			
			\draw[blue!60, line width=1.6pt] (0,0) -- (0,1) -- (0,2) -- (1,2) -- (1,3) -- (2,3) -- (3,3) -- (3,4) -- (3,5) -- (4,5) -- (4,6) -- (5,6) -- (5,7) -- (5,8) -- (8,8);
			
			\node at (.5,.5) {$1$};
			\draw (.5, .5) circle (.4cm);
			\node at (.5,1.5) {$2$};
			\draw (.5, 1.5) circle (.4cm);
			\node at (1.5,2.5) {$6$};
			\draw (1.5,2.5) circle (.4cm); 	
			\node at (3.5,3.5) {$0$};
			\draw (3.5, 3.5) circle (.4cm);
			\node at (3.5,4.5) {$3$};
			\draw (3.5, 4.5) circle (.4cm);
			\node at (4.5,5.5) {$4$};
			\draw (4.5, 5.5) circle (.4cm);
			\node at (5.5,6.5) {$0$};
			\draw (5.5, 6.5) circle (.4cm);
			\node at (5.5,7.5) {$5$};
			\draw (5.5, 7.5) circle (.4cm);
			\node at (1-1-0.5,1+0.5) {$\ast$};
			\node at (5-2-0.5,4+0.5) {$\ast$};
			
			\end{tikzpicture}
			\caption{Path in $\PLD(2,4)^{\ast 2}$ corresponding to the one in Figure~\ref{fig:pldExample2}.}	
			\label{fig:pldExample2-labelled}
		\end{center}
	\end{minipage}
\end{figure}

We define three statistics on $\DD(m,n)^{\ast a, \circ b}$.

The definition of the \emph{area} of a path in $\DD(m,n)^{\ast a, \circ b}$ is the same for a path in $\PLD(m,n)^{\ast a}$ (see Definition \ref{def: area DP}).

\begin{definition} \label{def: dinv unlabelled}
	For $D\in \DD(m,n)^{\ast a, \circ b}$. For $1\leq i \leq n+m$, we say that the pair $(i,j)$ is an inversion if 
	\begin{itemize}
		\item either $a_i(D) = a_j(D)$, $i \not \in \DPeak(D)$, and $j \not\in \ZVal(D)$ (\emph{primary inversion}),
		\item or $a_i(D) = a_j(D) + 1$, $j \not\in \DPeak(D)$, and  $i \not \in \ZVal(D)$ (\emph{secondary inversion}).
	\end{itemize}
	Then we define \[\dinv(D)\coloneqq \# \{0 \leq i < j \leq n+m \mid (i,j) \; \text{is an inversion}. \}\]
\end{definition}

For example, the path in Figure~\ref{fig:pldExample2} has dinv 6: 4 primary and 2 secondary. 

\begin{remark}
	Let $D \in \PLD(m,n)$. We define its \emph{dinv reading word} as the sequence of labels read starting from the ones in the main diagonal going bottom to top, left to right; next the ones in the diagonal $x+y=1$ bottom to top, left to right; then the ones in the diagonal $x+y=2$ and so on. 
	
	One can consider the paths in $\DD(m,n)^{\ast a, \circ b}$ as partially labelled decorated Dyck paths where the reading word is a shuffle of $m$ $0$'s, the string $1, \cdots, n-b$, and the string $n, \cdots n-b+1$. Indeed, given this restriction and the information about the position of the zero labels and considering the $b$ biggest labels to label the decorated peaks, the rest of the labelling is fixed. With regard to this labelling the Definitions \ref{def: dinv unlabelled} and \ref{def: dinv DP} of the dinv coincide.

	For example, the path in Figure~\ref{fig:pldExample2-labelled} is the partially labelled Dyck path corresponding to the decorated Dyck path in Figure~\ref{fig:pldExample2}. Indeed it has dinv reading word $10263405$ which is a shuffle of two $0$'s and the strings $1,2,3,4$ and $6,5$. Its dinv equals 6: 4 primary and 2 secondary. 
\end{remark}

Finally we define a third statistic on this set: the bounce. In order to do this, it is helpful to replace the zero valleys (the vertical step \emph{and} the horizontal step that precedes it) by diagonal steps. We construct the \emph{bounce path} as follows. Place a ball in the origin and send it travelling north. When it hits the beginning of a horizontal or diagonal step it changes direction and always travels parallel to the horizontal or diagonal step of the path in its current column, until it hits the main diagonal. There it turns north again. Repeat this process until the ball arrives at $(m+n,m+n)$. See Figure~\ref{fig: bounce path} for an example. 

Now label the vertical and diagonal steps of the bounce path starting with $i=0$'s and changing the label to $i+1$ each time the bounce path touches the main diagonal. Reading these labels bottom to top, we obtain the \emph{bounce word} of the path, denoted by $b_1(D)\dots b_{m+n}(D)$. See Figure~\ref{fig: bounce path} for an example.  

\begin{definition}
	Let $D\in \DD(m,n)^{\ast a , \circ b}$ and $b_1(D)\dots b_{n+m}(D)$ its bounce word. We define 
	\[\bounce(D)\coloneqq \sum_{i\not \in S}b_i(D), \] where $S$ is the set of the indices $i$ obtained by starting at a decorated peak and tracing a path to the east, parallel to the bounce path until this path hits the bounce path at the end of its $i$-th vertical or diagonal step. 
\end{definition}

For example, the path in Figure \ref{fig: bounce path} has bounce equal to $1$. 
\begin{figure}[!ht]
	\centering
	\begin{tikzpicture}[scale=0.5]
	\draw[step=1.0, gray!60, thin] (0,0) grid (12,12);
	
	\draw[gray!60, thin] (0,0) -- (12,12);
	
	\draw[blue!60, transform canvas={xshift=0.15mm}, transform canvas={yshift=-0.15mm}, line width=1.6pt] (0,0) -- (0,1) -- (0,2) -- (0,3) -- (1,3) -- (1,4) -- (2,4) -- (2,5) -- (3,5) -- (4,5) -- (4,6) -- (4,7) -- (4,8) -- (5,8) -- (6,8) -- (6,9) -- (6,10) -- (7,10) -- (7,11) -- (8,11) -- (8,12) -- (9,12) -- (10,12) -- (11,12) -- (12,12);
	
	\draw[blue!20, transform canvas={xshift=-0.15mm}, transform canvas={yshift=0.15mm}, line width=1.6pt] (0,0) -- (0,1) -- (0,2) -- (0,3) -- (1,4) -- (2,5) -- (3,5) -- (4,6) -- (4,7) -- (4,8) -- (5,8) -- (6,9) -- (6,10) -- (7,11) -- (8,12) -- (9,12) -- (10,12) -- (11,12) -- (12,12);

	\draw[dashed, opacity=0.6, thick] (0,0) -- (0,3) -- (1,4) -- (2,5) -- (3,5) -- (4,6) -- (5,6) -- (6,7) -- (7,8) -- (8,9) -- (9,9) -- (9,12) -- (10,12) -- (11,12) -- (12,12);
	
	\node at (1.5,3.5) {$0$};
	\draw (1.5, 3.5) circle (.4cm);
	\node at (2.5,4.5) {$0$};
	\draw (2.5, 4.5) circle (.4cm);
	\node at (4.5,5.5) {$0$};
	\draw (4.5, 5.5) circle (.4cm);
	\node at (6.5,8.5) {$0$};
	\draw (6.5, 8.5) circle (.4cm);
	\node at (7.5,10.5) {$0$};
	\draw (7.5, 10.5) circle (.4cm);
	\node at (8.5,11.5) {$0$};
	\draw (8.5, 11.5) circle (.4cm);
	
	\node at (-0.5,0.5) {$0$};
	\node at (-0.5,1.5) {$0$};
	\node at (-0.5,2.5) {$0$};
	\node at (-0.5,3.5) {$0$};
	\node at (-0.5,4.5) {$0$};
	\node at (-0.5,5.5) {$0$};
	\node at (-0.5,6.5) {$0$};
	\node at (-0.5,7.5) {$0$};
	\node at (-0.5,8.5) {$0$};
	\node at (-0.5,9.5) {$1$};
	\node at (-0.5,10.5) {$1$};
	\node at (-0.5,11.5) {$1$};
	
	\filldraw[fill=blue!60]
	(4,8) circle (3 pt)
	(6,10) circle (3 pt);
	\end{tikzpicture}
	\caption{Bounce path and bounce word (left) of a path in $\DD(6,6)^{\ast 0, \circ 2}$.} \label{fig: bounce path}
\end{figure}

\subsection{Recursion for $(\dinv, \area)$}

Define the subset \[\DDd(m,n\backslash r)^{\ast a, \circ b}\subseteq \DD(m,n)^{\ast a, \circ b}\] to consist of the paths $D\in \DD(m,n)^{\ast a, \circ b}$ such that \[\#\{1\leq i\leq n+m \mid a_i(D)=0 \,\text{and}\, i \not \in \ZVal(D) \} = r.\]	
We set \[\DDd_{q,t}(m,n\backslash r)^{\ast a, \circ b}\coloneqq \sum_{D\in \DDd(m,n\backslash r)^{\ast a, \circ b}}q^{\dinv(D)}t^{\area(D)}\]

\begin{theorem}\label{thm: recursion pld dinv}
	$\DDd_{q,t}(m,n\backslash r)^{\ast a, \circ b} = F^{(b,a)}_{n,r;m} $ .
\end{theorem}

\begin{proof}
	We will show that $\DDd_{q,t}(m,n\backslash r)^{\ast a, \circ b}$ satisfies the same recursion and initial conditions as in Corollary \ref{cor: iterated recursion}. In other words we will show that 
	\begin{align*}
		\DDd_{q,t}(m,n\backslash r)^{\ast a, \circ b} &= t^{n-r-a}\sum_{j=0}^m \sum_{s=0}^r q^{\binom{s}{2}} \qbinom{r}{s}_q\qbinom{r+j-1}{j}_q\\ & \times t^{m-j} \sum_{u=0}^{n-r-a}\sum_{v=0}^{s+j}q^{\binom{v}{2}} \qbinom{s+j}{v}_q\qbinom{s+j+u-1}{u}_q \\ & \times \DDd_{q,t}(m-j,n-r\backslash u+v)^{\ast a-v, \circ b-(r-s)}
	\end{align*} 	
	and \[\DDd_{q,t}(m,n\backslash n) = \delta_{a,0} q^{n-b\choose 2} \qbinom{m+n-1}{m}_q \qbinom{n}{b}_q.\]
	
	Let us start with the second identity. The set $\DDd(m,n\backslash n)$ consists of the paths whose area word contains only 0's, indeed any valley that is not at height 0 is preceded by a vertical step that is not a valley. Thus the area must be zero. Furthermore there can be no rises, which explains $\delta_{a,0}$. The dinv between steps that are not zero valleys nor decorated peaks is counted by $q^{n-b\choose 2 }$. The dinv between zero valleys and things that are not zero valleys are counted by $\qbinom{m+n-1}{m}_q$ because we are not allowed to start with a zero valley. Finally, the dinv between peaks and steps that are not zero valleys is taken into account by $\qbinom{n}{b}_q$. 
	
	Now for the recursive step. We give an overview of the combinatorial interpretations of all the variables appearing in this formula. We say that a vertical step of a path is \emph{at height $i$} if its corresponding letter in the area word equals $i$.   
	
	\begin{itemize}
		\item $r$ is the number of zeroes in the area word whose index is not a zero valley. 
		
		\item $r-s$ is the number of decorated peaks at height 0.
		\item The previous two imply that $s$ is the number of zeroes in the area word whose index in not a decorated peak nor a zero valley.
		\item $j$ is the number of zero valleys at height 0.
		\item $v$ is the number of decorated rises at height 1.
		\item $u+v$ is the number of $1$'s in the area word whose index is not a zero valley. 
	\end{itemize}
	
	The strategy of this recursion is the following. Start from a path $D$ in $\DDd(m,n\backslash r)^{\ast a, \circ b}$. Remove all the $0$'s from the area word, and then remove both the corresponding decoration on peaks, and decorations on rises at height one (which are not rises any more). Then decrease all the remaining letters by $1$. In this way we obtain a path in  \[\DDd(m-j,n-r\backslash u+v)^{\ast a-v, \circ b-(r-s)}.\]
	
	Let us look at what happens to the statistics of the path. 
	
	The area goes down by the size (i.e. $m+n$), minus the number of zeroes in the area word (i.e. $r+j$) and the number of rises (i.e. $a$), since these letters did not contribute to the area to begin with. This explains the term $t^{m+n-(r+j+a)}$.   
	
	The factor $q^{\binom{s}{2}}$ takes into account the primary dinv among $0$'s that are neither zero valleys nor decorated peaks. The factor $\qbinom{r}{s}_q$ takes into account the primary dinv among $0$'s that are neither zero valleys nor decorated peaks, and $0$'s that are decorated peaks. Indeed, each time a one of the former precedes one of the latter one unit of primary dinv is created. The factor $\qbinom{r+j-1}{j}_q$ takes into account the primary dinv among $0$'s that are zero valleys and the other $0$'s, where we get $r-1$ because the first $0$ cannot be a zero valley.
	
	The factor $q^{\binom{v}{2}}$ takes into account the secondary dinv among $1$'s that are decorated rises and $0$'s that are directly below a decorated rise. The factor $\qbinom{j+s}{v}_q$ takes into account the secondary dinv among those $1$'s, and the remaining $0$'s that are not decorated peaks. The factor $\qbinom{j+s+u-1}{u}_q$ takes into account the secondary among the remaining $1$'s and the $0$'s that are not decorated peaks, where we get $s-1$ because the first non-decorated peak $0$ must be before the first $1$.
	
	Summing over all the possible values of $j,s,u$ and $v$, we obtain the stated recursion.
\end{proof}

\subsection{Recursion for $(\area, \bounce)$}

\begin{definition}
	A \emph{fall} of a Dyck path is a horizontal step followed by another horizontal step. 
\end{definition}

There exists a natural map, mapping rises into falls. Indeed the joining point of the two vertical steps of a rise is a point where the path vertically crosses a certain diagonal parallel to the main diagonal. Since the path must end at the main diagonal, it must cross the same diagonal horizontally with a fall at least once. We map the rise to the first of such falls: this clearly yields a bijective map. Furthermore, the number of squares between the path an the main diagonal in the row of the decorated rise is equal to the number of squares between the path and the main diagonal in the column of the corresponding decorated fall (see Figure~\ref{fig:rises-falls-correspondence}). Thus, equivalent to the definition of the area in Definition~\ref{def: area DP} is the number of whole squares that lie between the path and the main diagonal, except the ones in the columns containing falls. 

\begin{figure}[!ht]
	\centering
	\begin{tikzpicture}[scale=0.6]
	\draw[gray!60, thin] (0,0) grid (11,11) (0,0) -- (11,11) (-1,1) -- (10,12);
	\draw[blue!60, line width = 1.6pt] (0,0) -- (0,2) -- (1,2) -- (1,3) -- (2,3) -- (2,6) -- (4,6) -- (4,9) -- (5,9) -- (5,10) -- (7,10) -- (7,11) -- (11,11);
	\draw[line width = 1.6pt] (2,3) -- (2,5) (8,11) -- (10,11);
	\fill[pattern=north west lines, pattern color=gray] (2,4) rectangle (4,5) (8,11) rectangle (9,9);
	\draw  (1.5,4.5) node {$\ast$};
	\draw (8.5,11.5) node {$\ast$};
	\end{tikzpicture} 
	\caption{Correspondence between rises and falls.}
	\label{fig:rises-falls-correspondence}
\end{figure}
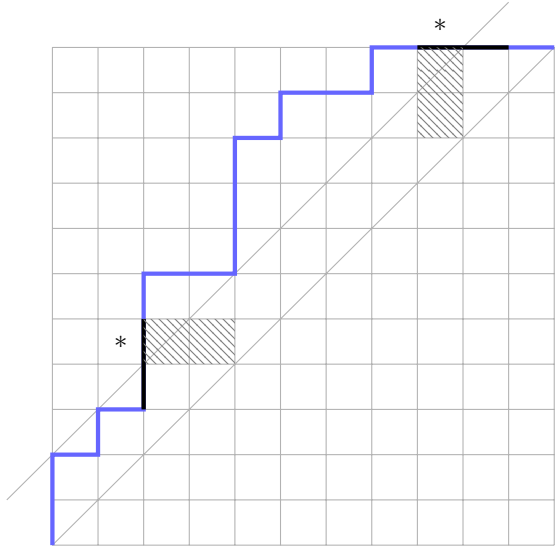

This time, we define the set \[\DDb(m,n\backslash r)^{\ast a, \circ b}\] to consist of the paths $D$ 

\begin{enumerate}[(i)]
	\item with $m$ zero valleys;
	\item $b$ decorated peaks that cannot be zero valleys and where we \emph{do not allow the leftmost peak to be decorated};
	\item $n$ vertical steps that are not zero valleys;
	\item $a$ decorated falls \emph{where we allow the last step of the path to be a decorated fall};
	\item starting with $r$ vertical steps followed by a horizontal step.
\end{enumerate}

We set the polynomial \[\DDb_{q,t}(m,n\backslash r)^{\ast a, \circ b}\coloneqq \sum_{D\in \DDb(m,n\backslash r)^{\ast a, \circ b}}q^{\area(D)}t^{\bounce(D)}\]

\begin{theorem}\label{thm: recursion pld bounce}
	$\DDb_{q,t}(m,n\backslash r)^{\ast b, \circ a} = F^{(a,b)}_{n,r;m} $ 
\end{theorem}

\begin{proof}
	We will show that $\DDb_{q,t}(m,n\backslash r)^{\ast b, \circ a}$ satisfies the same recursion and initial conditions as in Corollary \ref{cor: iterated recursion}. In other words we will show that 
	\begin{align*}
		\DDb_{q,t}(m,n\backslash r)^{\ast b, \circ a} &= t^{n-r-a}\sum_{j=0}^m \sum_{s=0}^r q^{\binom{s}{2}} \qbinom{r}{s}_q\qbinom{r+j-1}{j}_q\\ & \times t^{m-j} \sum_{u=0}^{n-r-a}\sum_{v=0}^{s+j}q^{\binom{v}{2}} \qbinom{s+j}{v}_q\qbinom{s+j+u-1}{u}_q \\ & \times \DDb_{q,t}(m-j,n-r\backslash u+v)^{\ast b-(r-s), \circ  a-v}.
	\end{align*} 	
	
	and \[\DDb_{q,t}(m,n\backslash n)^{\ast b, \circ a} = \delta_{a,0} q^{n-b\choose 2} \qbinom{m+n-1}{m}_q \qbinom{n}{b}_q.\]
	
	Let us start with the second identity. The set $\DDb_{q,t}(m,n\backslash n)^{\ast b, \circ a}$ consists of the paths starting with $n$ vertical steps followed by only horizontal steps and zero valleys. Since zero valleys cannot be decorated peak and the leftmost peak is not allowed to be decorated, this implies that there can be no decorated peaks, which explains the factor $\delta_{a,0}$. Now think of the decorated falls as their corresponding decorated rises. The area east of the $n$ first vertical steps is counted by $q^{n-b\choose 2} \qbinom{n}{b}$; we use the fact that \emph{any} of these $n$ vertical steps may be a decorated rise, since we allowed the last step of the path to be a fall. Finally, the area west of the rest of the path is counted by $\qbinom{m+n-1}{m}_q$, the interlacing of the zero valleys and the horizontal steps, where we must end with a horizontal step. 
	
	From \cite{DAdderio-VandenWyngaerd-2017}*{Lemma~2.13}, we can deduce (replacing $k$ by $s$, $s$ by $h$ and $h$ by $u$) that 
	
	\begin{align*}
		\qbinom{j+s}{v}_q \qbinom{j+s+u-1}{u}_q & = q^v \qbinom{u+v-1}{v}_q \qbinom{u+s+j-1}{u+v}_q 
		+ \qbinom{u+v-1}{v-1}_q \qbinom{u+j+s}{u+v}_q.
	\end{align*}
	
	Multiplying this equation by $q^{\binom{v}{2}}$ and setting $u= h-v$ we get
	
	\begin{align*}	
		q^{\binom{v}{2}} \qbinom{j+s}{v}_q \qbinom{j+s+u-1}{u}_q &= q^{\binom{v+1}{2}} \qbinom{h-1}{v}_q \qbinom{h-v+s+j-1}{h}_q 
		\\&+ q^{\binom{v}{2}}\qbinom{h-1}{v-1}_q \qbinom{h-v+s+j}{h}_q 
	\end{align*}

	and we see that the second term is just the first where $v$ is replaced by $v-1$. Using this and regrouping some terms, we can rewrite the recursion
	\begin{align*}
		\DDb_{q,t}(m,n\backslash r)^{\ast b, \circ a} &= t^{n-r-a}\sum_{j=0}^m \sum_{s=0}^r q^{s\choose 2 } \qbinom{r}{s}_q\qbinom{r+j-1}{j}_q\\ 
		& \times t^{m-j} \sum_{h=0}^{n-r-a}\sum_{v=0}^{s+j} q^{\binom{v+1}{2}} \qbinom{h-1}{v}_q \qbinom{h-v+s+j-1}{h}_q \\
		& \times (\DDb(m-j,n-r\backslash h)^{\ast b-(r-s) , \circ a-v}\\&+\DDb(m-j,n-r\backslash h)^{\ast b-(r-s) , \circ a-(v-1)}).
	\end{align*} 
	
	As when we defined the bounce, we will think of blank valleys (the vertical step \emph{and} the preceding horizontal step) as diagonal steps. 
	
	Let us again start with a list of combinatorial interpretations for the variables appearing in this formula. 
	
	\begin{itemize}
		\item $r$ is the number of vertical steps in the first column. 
		\item $r-s$ is the number of decorated falls above the first bounce. 
		\item We can conclude from the previous two that $s$ is the number of horizontal steps above the first bounce that are not decorated falls. 
		\item $j$ is the number of zero valleys above the first bounce.
		\item $v$ is the number of decorated peaks above the first bounce.
		\item $h$ is the number of vertical steps in the column where the bounce path first hits the main diagonal. 
	\end{itemize}
	
	The strategy of this recursion is the following. Start from a path $D$ in $\DDb(m,n\backslash r)^{\ast b, \circ a}$. Remove the $r+j$ first rows and columns of the path. The remaining path goes from $(r+j, r+j+h)$ to $(m+n,m+n)$. Add $h$ vertical steps to the beginning of this path an we obtain a Dyck path $D'$ of size $m+n-(r+j)$. See Figure \ref{fig: bouncerec} for an illustration. Call $P(D)$ the part of $D$ that gets deleted (i.e. the path from $(0,0)$ to $(r+j,r+j+h)$).

	\begin{figure}[!ht]
		\begin{center}
			\begin{tikzpicture}[scale=.6]
			\draw[step=1.0,gray,opacity=0.6,thin] (0,0) grid (12,12) (0,0)--(12,12);
			\draw[blue!60, transform canvas={xshift=0.15mm}, transform canvas={yshift=-0.15mm}, line width=1.6pt](0,0)|-(1,4)|-(2,6) (4,8)|-(5,8)(6,9)|-(7,10)(8,11)--(11,11)|-(12,12);
			\draw[blue!60, transform canvas={xshift=0.15mm}, transform canvas={yshift=-0.15mm}, line width=1.6pt](2,6)-|(3,7) -|(4,8)(5,8)-|(6,9)(7,10)-|(8,11); 
			\draw[blue!20, transform canvas={xshift=-0.15mm}, transform canvas={yshift=0.15mm}, line width=1.6pt] (0,0)|-(1,4)|-(2,6)--(4,8) (4,8)|-(5,8)--(6,9)|-(7,10)--(8,11)--(11,11)|-(12,12);
			\draw (3.5,6.5) node {0} (4.5,7.5) node {0} (6.5,8.5) node {0} (8.5,10.5) node {0};
			\draw (3.5,6.5) circle (.4cm) (4.5,7.5) circle (.4cm) (6.5,8.5) circle (.4cm) (8.5,10.5) circle (.4cm);
			\draw[dashed, opacity=0.6, thick](0,0)--(0,4)--(2,4)--(4,6)--(5,6)--(6,7)--(7,7)--(7,10)--(8,11)--(11,11)|-(12,12);
			\draw[gray](6.9,6.9) rectangle (12.4,12.4);
			\draw (1.5, 6.5) node {$\ast$} (4.5,8.5) node {$\ast$} (8.5, 11.5) node {$\ast$};
			\filldraw[fill=blue!60] (1,6) circle (3pt) (11,12) circle (3pt);
			\draw (13,.5) node {0}(13,1.5) node {0}(13,2.5) node {0}(13,3.5) node {0}(13,4.5) node {0}(13,5.5) node {0}(13,6.5) node {0}(13,7.5) node {1}(13,8.5) node {\textcolor{red}{1}}(13,9.5) node {1}(13,10.5) node {1}(13,11.5) node {\textcolor{red}{2}};
			\end{tikzpicture}
		\end{center}	
		\caption{Bounce recursion} \label{fig: bouncerec}	
	\end{figure}
	
	Note that an ambiguous situation arises when the last vertical step of $P(D)$ is part of a decorated peak, since the leftmost peak of $D'$ is never decorated. We solve this by considering two terms for the reduced path, in the first the number of decorated peaks gets reduced by $v$ and the last peak of $P(D)$ is \emph{never} decorated and in the second the number of decorated peaks gets reduced by $v-1$ and the last peak of $P(D)$ is \emph{always} decorated. So $D'$ is a path either in $\DDb(m-j,n-r\backslash h)^{\ast b-(r-s) , \circ a-v}$ or in $\DDb(m-j,n-r\backslash h)^{\ast b-(r-s) , \circ a-(v-1)}$.
	
	Going from $D$ to $D'$, the bounce goes down by the size (i.e. $m+n$) of the path, minus the number of zeroes in the bounce word (i.e. $r+j$) and the number of decorated peaks of $D$ (i.e. $a$), since these letters of the bounce word did not contribute to the bounce to begin with. This explains the term $t^{n+m-(r+j+a)}$. 
	
	The area goes down by the number of squares under $P(D)$ contributing to the area of $D$. 
	
	First consider the whole squares under the bounce path. There are $r$ horizontal steps in this section of the bounce path. Since there are no vertical steps the $(r-i)$-th horizontal step of the bounce path has exactly $i-1$ squares of area under it (and above the main diagonal). Some of these however, do not contribute to area because they are under a decorated fall. Since $r-s$ is the number of decorated falls over the first bounce, $s$ is the number of horizontal steps in the first bounce of the bounce path that do not lie under decorated falls. So choosing $s$ different values between $0$ and $r-1$ yields the term $q^{\binom{s}{2}} \qbinom{r}{s}_q$. The positioning of the $j$ diagonal steps (which cannot lie under falls) among $r-1$ horizontal steps (indeed the last step must be horizontal) explains the term $\qbinom{r+j-1}{j}_q$.
	
	Next, consider the area under $P(D)$ and above the bounce path. Choosing which of the $h$ vertical steps of $P(D)$ are decorated peaks and inserting the vertical step that follows it gives a factor of $q^{v+1\choose 2}\qbinom{h-1}{v}_q$, indeed we require that the last vertical step is not a peak. Finally, we need to choose an interlacing between the $s+j-v$ horizontal steps that are not decorated falls and do not follow a decorated peak \emph{or} diagonal steps and the $h$ vertical steps that are not decorated peaks, which gives $\qbinom{h-v+s+j-1}{h}_q$ because the first step cannot be vertical.

\end{proof}

\subsection{A statistics swapping bijection}

Let $\DDb(m, n \backslash r)^{\triangle b, \circ a}$ be the set of decorated Dyck paths of size $m+n$ with $m$ zero valleys, $a$ decorated peaks, and $b$ decorated \emph{fake falls}, where by fake fall we mean any horizontal step that is not in the same column of a zero valley, and that is immediately followed by either another horizontal step or a vertical step that is both a zero valley and a peak. The last horizontal step (if it is not in the same column of a zero valley) is also a fake fall. We can define statistics $\area$ and $\bounce$ on this set as usual, not counting the squares below decorated fake falls, and computing the bounce the same way we did for $\DDb(m, n \backslash r)^{\ast b, \circ a}$. We have the following result.

\begin{theorem}
	There is a bijection between $\DDd(m, n \backslash r)^{\ast a, \circ b}$ and $\DDb(m, n \backslash r)^{\triangle b, \circ a}$ mapping $(\dinv, \area)$ to $(\area, \bounce)$.
\end{theorem}

This map generalizes the classical sweep map on Dyck paths.

\begin{proof}
	Let $D \in \DDd(m, n \backslash r)^{\ast a, \circ b}$, and consider its area word. We draw a Dyck path as follows. We scan the area word left to right and draw a vertical step for each $0$ that is not a zero valley. Then, for $i=0,1,2,\dots$, we scan the area word again and draw a horizontal step for each $i$ that is not a zero valley, a vertical step followed by a horizontal step for each $i$ that is a zero valley, and a vertical step for each $i+1$ that is not a zero valley. We repeat the procedure until there are no more letters in the area word. It is easy to check that the resulting path is a Dyck path. See Figure~\ref{fig:sweep} for an example.
	
	By construction, zero valleys are mapped to a subset of the peaks (since we drew a peak every time we scanned one). We put a zero valley in the first vertical step of each column containing one of the peaks that came from a zero valley; this way, the number of zero valleys is preserved.
	
	We also have that peaks that are not zero valleys are mapped into fake falls. In fact we have such a peak whenever we scan two $i$'s with no $i+1$ in between, and the first $i$ is not a zero valley. In the image, those peaks correspond to horizontal steps followed either by another horizontal step (if the second $i$ is not a zero valley), or by a vertical step that is both a zero valley and a peak (if the second $i$ is a zero valley). The topmost, rightmost peak correspond to the last horizontal step (there is no $i+1$), which is also a fake fall. Whenever such a peak is decorated, we decorate the corresponding fake fall in the image.
	
	Finally, rises are mapped into valleys. In fact we have a rise whenever we scan an $i$ followed by an $i+1$, which means that we draw a horizontal step (possibly preceded by a vertical step if the $i$ is a zero valley) followed by a vertical step, which is exactly the definition of valley. Whenever such a rise is decorated, we decorate the peak in the same column as the corresponding valley in the image. It means that the leftmost peak is never decorated, since it doesn't correspond to a valley.
	
	The process is straightforward to revert, since the area word of the starting path is completely determined by the interlacing between horizontal and vertical steps of the image, and the same is true for the position of the decorations. Hence the map is bijective.
	
	We now have to prove that the statistics are preserved. It is an easy check that the bounce word of the image is an anagram of the area word of the starting Dyck path, where the letters that are not zero valleys correspond to vertical steps, and the zero valleys correspond to diagonal steps. Furthermore, if a letter is a decorated rise, the corresponding letter of the bounce word of the image will be cancelled by the corresponding decorated peak. 
	
	The area below each horizontal step in the image is exactly the primary dinv on the right and the secondary dinv on the left given by the corresponding letter of the area word of the starting path. In fact, the height of that horizontal step with respect to the diagonal is the number of vertical steps before it minus the number of horizontal steps before it, which is the number of $i$'s in the original area word on its left that are not zero valleys, plus the number of $i+1$'s in the original area word on its right that are not zero valleys. The primary dinv is the area below the bounce path (i.e. the number of $i$'s in the original area word on its left that are not zero valleys), the secondary is the area above it (i.e. the number of $i+1$'s in the original area word on its right that are not zero valleys - the area in a square crossed diagonally by the bounce path counts as above). 
	If that letter is a decorated peak, then the corresponding horizontal step will be a decorated fake fall, so the dinv in the preimage and to the area in the image drop by the same amount.
	
	The thesis follows.
\end{proof}

\begin{figure}[!ht]
	\begin{minipage}{0.5\textwidth}
		\begin{center}
			\begin{tikzpicture}[scale=0.5]
			\draw[gray!60, thin] (0,0) grid (12,12);
			
			\draw[gray!60, thin] (0,0) -- (12,12);
			
			\draw[blue!60, line width= 1.6pt] (0,0) -- (0,1) -- (0,2) -- (1,2) -- (1,3) -- (2,3) -- (2,4) -- (3,4) -- (4,4) -- (4,5) -- (5,5) -- (5,6) -- (6,6) -- (6,7) -- (7,7) -- (7,8) -- (8,8) -- (8,9) -- (8,10) -- (8,11) -- (9,11) -- (10,11) -- (11,11) -- (11,12) -- (12,12);
			
			\draw
			(6.5,6.5) node {$0$} circle (.4cm)
			(7.5,7.5) node {$0$} circle (.4cm)
			(1.5,2.5) node {$0$} circle (.4cm)
			(11.5,11.5) node {$0$} circle (.4cm);
			
			\draw
			(0,1.5) node[left, black] {$\ast$}
			(8,10.5) node[left, lime] {$\ast$};
			
			\filldraw[fill=cyan] (5,6) circle (3pt);
			\filldraw[fill=red]	(2,4) circle (3pt);
			
			\node at (13, 0.5) {$0$};
			\node at (13, 1.5) {$1$};
			\node at (13, 2.5) {$1$};
			\node at (13, 3.5) {$1$};
			\node at (13, 4.5) {$0$};
			\node at (13, 5.5) {$0$};
			\node at (13, 6.5) {$0$};
			\node at (13, 7.5) {$0$};
			\node at (13, 8.5) {$0$};
			\node at (13, 9.5) {$1$};
			\node at (13, 10.5) {$2$};
			\node at (13, 11.5) {$0$};
			
			\end{tikzpicture}
		\end{center}
	\end{minipage}%
	\begin{minipage}{0.5\textwidth}
		\begin{center}
			\begin{tikzpicture}[scale=0.5]
			\draw[step=1.0,gray,opacity=0.6,thin] (0,0) grid (12,12) (0,0)--(12,12);
			
			\draw[blue!60, transform canvas={xshift=0.15mm}, transform canvas={yshift=-0.15mm}, line width=1.6pt](0,0)|-(1,4)|-(2,6)-|(3,7)-|(4,8)-|(5,8)-|(6,9)|-(7,10)-|(8,11)--(11,11)|-(12,12);
			
			\draw
			(3.5,6.5) node {0} circle (.4cm)
			(4.5,7.5) node {0} circle (.4cm)
			(6.5,8.5) node {0} circle (.4cm)
			(8.5,10.5) node {0} circle (.4cm);
			
			\draw[dashed, opacity=0.6, thick] (0,0) -- (0,4) -- (2,4) -- (4,6) -- (5,6) -- (6,7) -- (7,7) -- (7,10) -- (8,11) -- (11,11) |- (12,12);
			
			\draw
			(2.5,6) node[above, red] {$\ast$}
			(9.5,11) node[above, cyan] {$\ast$};
			
			\filldraw[fill=black] (1,6) circle (3pt);
			\filldraw[fill=lime] (11,12) circle (3pt);
			\end{tikzpicture}
		\end{center}
	\end{minipage}	
	\caption{A path $D \in \DDd(m, n \backslash r)^{\ast a, \circ b}$ with its area word shown (left) and its image (right).} 
	\label{fig:sweep}
\end{figure}

\section{Doubly decorated polyominoes}

In \cite{DAdderio-Iraci-polyominoes-2017}*{Section~3} the authors define decorated reduced parallelogram polyominoes. 

\begin{definition}[\cite{DAdderio-Iraci-polyominoes-2017}*{Definition~3.1}] 
	\label{def:reducedpolyominoes}
	A \emph{reduced polyomino} of size $m \times n$ is a pair of lattice paths from $(0,0)$ to $(m,n)$ using only north and east steps, such that the first one (the \emph{red path}) lies always weakly above the second one (the \emph{green path}).
\end{definition}

The set of reduced polyominoes of size $m \times n$ is denoted by $\RP(m,n)$. Reduced polyominoes are encoded by their area word.

\begin{definition}
	An \emph{area word} is a (finite) string of symbols $a_1 a_2 \cdots a_n$ in a well-ordered alphabet such that if $a_i < a_{i+1}$ then $a_{i+1}$ is the successor of $a_i$ in the alphabet.
\end{definition}

The area word of a reduced polyomino is an area word in the alphabet $\overline{\mathbb{N}} \coloneqq 0 < \bar{0} < 1 < \bar{1} < 2 < \dots$ starting with $0$.

It is computed in the following way. The first step consists of drawing a diagonal of slope $-1$ from the end of every horizontal green step, and attaching to that step the length of that diagonal (i.e. the number of squares it crosses, that can also be zero). Then, one puts a dot in every square not crossed by any of those diagonals, and attaches to each vertical red step the number of dots in the corresponding row. Next, one bars the numbers attached to vertical red steps. The area word starts (artificially) with a $0$, then one reads those numbers following the diagonals of slope $-1$, writing down the labels when encountering the end of its step and the red label before the green one. The correspondence is bijective. See Figure~\ref{fig:aw-reduced-polyomino} for an example.

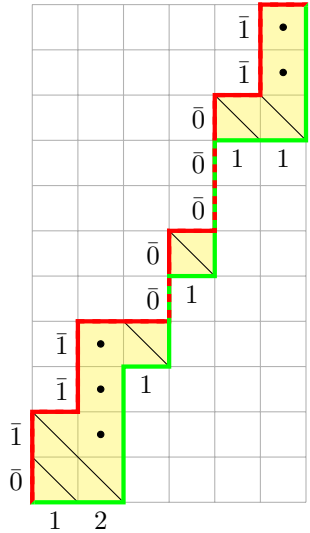
\begin{figure}[!ht]
	\centering
	\begin{tikzpicture}[scale=0.6]
	\draw[gray!60, thin] (0,0) grid (6,11);
	
	\filldraw[yellow, opacity=0.3] (0,0) -- (1,0) -- (2,0) -- (2,1) -- (2,2) -- (2,3) -- (3,3) -- (3,4) -- (3,5) -- (4,5) -- (4,6) -- (4,7) -- (4,8) -- (5,8) -- (6,8) -- (6,9) -- (6,10) -- (6,11) -- (5,11) -- (5,10) -- (5,9) -- (4,9) -- (4,8) -- (4,7) -- (4,6) -- (3,6) -- (3,5) -- (3,4) -- (2,4) -- (1,4) -- (1,3) -- (1,2) -- (0,2) -- (0,1) -- (0,0);
	
	\draw
	(1,0) -- (0,1)
	(2,0) -- (0,2)
	(3,3) -- (2,4)
	(4,5) -- (3,6)
	(5,8) -- (4,9)
	(6,8) -- (5,9);
	
	\filldraw[fill=black]
	(1.5,1.5) circle (2pt)
	(1.5,2.5) circle (2pt)
	(1.5,3.5) circle (2pt)
	(5.5,9.5) circle (2pt)
	(5.5,10.5) circle (2pt);
	
	\node[below] at (0.5,0) {$1$};
	\node[below] at (1.5,0) {$2$};
	\node[below] at (2.5,3) {$1$};
	\node[below] at (3.5,5) {$1$};
	\node[below] at (4.5,8) {$1$};
	\node[below] at (5.5,8) {$1$};
	
	\node[left] at (0,0.5) {$\bar{0}$};
	\node[left] at (0,1.5) {$\bar{1}$};
	\node[left] at (1,2.5) {$\bar{1}$};
	\node[left] at (1,3.5) {$\bar{1}$};
	\node[left] at (3,4.5) {$\bar{0}$};
	\node[left] at (3,5.5) {$\bar{0}$};
	\node[left] at (4,6.5) {$\bar{0}$};
	\node[left] at (4,7.5) {$\bar{0}$};
	\node[left] at (4,8.5) {$\bar{0}$};
	\node[left] at (5,9.5) {$\bar{1}$};
	\node[left] at (5,10.5) {$\bar{1}$};
	
	\draw[red, sharp <-sharp >, sharp angle = -45, line width=1.6pt] (0,0) -- (0,1) -- (0,2) -- (1,2) -- (1,3) -- (1,4) -- (2,4) -- (3,4) -- (3,5) -- (3,6) -- (4,6) -- (4,7) -- (4,8) -- (4,9) -- (5,9) -- (5,10) -- (5,11) -- (6,11);
	
	\draw[green, sharp <-sharp >, sharp angle = 45, line width=1.6pt] (0,0) -- (1,0) -- (2,0) -- (2,1) -- (2,2) -- (2,3) -- (3,3) -- (3,4) -- (3,5) -- (4,5) -- (4,6) -- (4,7) -- (4,8) -- (5,8) -- (6,8) -- (6,9) -- (6,10) -- (6,11);
	
	\draw[red, dashed, sharp <-sharp >, sharp angle = -45, line width=1.6pt] (0,0) -- (0,1) -- (0,2) -- (1,2) -- (1,3) -- (1,4) -- (2,4) -- (3,4) -- (3,5) -- (3,6) -- (4,6) -- (4,7) -- (4,8) -- (4,9) -- (5,9) -- (5,10) -- (5,11) -- (6,11);
	\end{tikzpicture}
	
	\caption{A $6 \times 11$ reduced polyomino. Its area word is $0 \bar{0} 1 \bar{1} 2 \bar{1} \bar{1} 1 \bar{0} \bar{0} 1 \bar{0} \bar{0} \bar{0} 1 1 \bar{1} \bar{1}$.}
	\label{fig:aw-reduced-polyomino}
\end{figure}

\begin{remark}
	The area word of a reduced polyomino of size $m \times n$ has $m+1$ unbarred letters (including the starting $0$) and $n$ barred letters.
\end{remark}

In \cite{DAdderio-Iraci-polyominoes-2017}*{Section~3} two kind of decorations were introduced. We introduce now two more, together with more general sets of decorated reduced polyominoes. 

\begin{definition}[\cite{DAdderio-Iraci-polyominoes-2017}*{Definition~3.3}] 
	The \emph{unbarred rises} of a reduced polyomino $P$ are the letters of the area word that are unbarred, and that are the successor, in the alphabet $\overline{\mathbb{N}}$, of the letter immediately to their left.
\end{definition}

\begin{definition}
	The \emph{barred rises} of a reduced polyomino $P$ are the letters of the area word that are barred, and that are the successor, in the alphabet $\overline{\mathbb{N}}$, of the letter immediately to their left.
\end{definition}

We denote by $\RP(m,n)^{\ast k, j}$ the set of reduced polyominoes of size $m \times n$ with $k$ decorated unbarred rises and $j$ decorated barred rises.

\begin{definition}[\cite{DAdderio-Iraci-polyominoes-2017}*{Definition~3.5}] 
	The \emph{green peaks} of a reduced polyomino $P$ are the horizontal green steps that immediately follow a vertical green step.
\end{definition}

\begin{definition}
	The \emph{red valleys} of a reduced polyomino $P$ are the vertical red steps that immediately follow a horizontal red step. The first red step counts as a           valley if it is vertical.
\end{definition}

We denote by $\RP(m,n)^{\circ k, j}$ the set of reduced polyominoes of size $m \times n$ with $k$ decorated green peaks and $j$ decorated red valleys. We also identify $\RP(m,n) = \RP(m,n)^{\ast 0, 0} = \RP(m,n)^{\circ 0, 0}$. 

Let us fix some more notation.
\begin{align}
	&\RP(m,n) && \coloneqq \; \; \{ P \mid P \; \text{is a $m \times n$ reduced polyomino} \} \\
	&\RP(m \backslash r, n)^{\ast 0, 0} && \coloneqq \; \; \{ P \in \RP(m,n) \mid P \; \text{has $r$ $0$'s in area word} \} \label{mr} \\
	&\RP(m \backslash r, n)^{\ast k, j} && \coloneqq \; \; \{ P \in \RP(m,n)^{\ast k, j} \mid P \; \text{has $r$ $0$'s in area word}  \} \\
	&\RP(m \backslash r, n)^{\circ 0, 0} && \coloneqq \; \; \{ P \in \RP(m,n) \mid P \; \text{has $r\!-\!1$ $0$'s in bounce word} \} \label{ns} \\
	&\RP(m \backslash r, n)^{\circ k, j} && \coloneqq \; \; \{ P \in \RP(m,n)^{\circ k, j} \mid P \; \text{has $r\!-\!1$ $0$'s in b. word}  \}
\end{align}

Notice that we are including the first, artificial $0$ in \eqref{mr} and that we replaced $r$ with $r-1$ in \eqref{ns}. In particular, $r$ can be equal to $m+1$.

We now define three statistics on reduced polyominoes, namely \emph{area}, \emph{bounce}, and \emph{dinv}.

\begin{definition}
	For a doubly decorated reduced polyomino $P \in \RP(m,n)^{\ast k, j}$ we let $\area(P)$ be the sum of the letters in its area word (disregarding the bars), not counting decorated rises of either type.
\end{definition}

We now want to define the \emph{bounce path} of a reduced polyomino (see also \cite{DAdderio-Iraci-polyominoes-2017}*{Section~3}). It is a lattice path that starts from $(0,0)$ and goes east until it hits the beginning of a vertical green step, then it goes north until it hits the beginning of a red horizontal steps, then it bounces every time it hits the beginning of a step of any of the two paths. The labelling of the bounce path starts from $0$ and goes up by one in the alphabet $\overline{\mathbb{N}}$ every time it changes direction. 

\begin{figure}[!ht]
	\centering
	\begin{tikzpicture}[scale=0.6]
	\draw[gray!60, thin] (0,0) grid (6,11);
	
	\filldraw[yellow, opacity=0.3] (0,0) -- (1,0) -- (2,0) -- (2,1) -- (2,2) -- (2,3) -- (3,3) -- (3,4) -- (3,5) -- (4,5) -- (4,6) -- (4,7) -- (4,8) -- (5,8) -- (6,8) -- (6,9) -- (6,10) -- (6,11) -- (5,11) -- (5,10) -- (5,9) -- (4,9) -- (4,8) -- (4,7) -- (4,6) -- (3,6) -- (3,5) -- (3,4) -- (2,4) -- (1,4) -- (1,3) -- (1,2) -- (0,2) -- (0,1) -- (0,0);
	
	\draw[red, sharp <-sharp >, sharp angle = -45, line width=1.6pt] (0,0) -- (0,1) -- (0,2) -- (1,2) -- (1,3) -- (1,4) -- (2,4) -- (3,4) -- (3,5) -- (3,6) -- (4,6) -- (4,7) -- (4,8) -- (4,9) -- (5,9) -- (5,10) -- (5,11) -- (6,11);
	
	\draw[green, sharp <-sharp >, sharp angle = 45, line width=1.6pt] (0,0) -- (1,0) -- (2,0) -- (2,1) -- (2,2) -- (2,3) -- (3,3) -- (3,4) -- (3,5) -- (4,5) -- (4,6) -- (4,7) -- (4,8) -- (5,8) -- (6,8) -- (6,9) -- (6,10) -- (6,11);
	
	\draw[red, sharp <-sharp >, sharp angle = -45, dashed, line width=1.6pt] (0,0) -- (0,1) -- (0,2) -- (1,2) -- (1,3) -- (1,4) -- (2,4) -- (3,4) -- (3,5) -- (3,6) -- (4,6) -- (4,7) -- (4,8) -- (4,9) -- (5,9) -- (5,10) -- (5,11) -- (6,11);
	
	\draw[dashed, opacity=0.6, thick] (0,0) -- (1,0) -- (2,0) -- (2,1) -- (2,2) -- (2,3) -- (2,4) -- (3,4) -- (3,5) -- (3,6) -- (4,6) -- (4,7) -- (4,8) -- (4,9) -- (5,9) -- (6,9) -- (6,10) -- (6,11);
	
	\filldraw[fill=red]
	(1,2) circle (3pt)
	(4,6) circle (3pt);
	
	\filldraw[fill=green]
	(2,3) circle (3pt)
	(4,8) circle (3pt);
	
	\node[above] at (0.5,0.0) {$0$};
	\node[above] at (1.5,0.0) {$0$};
	\node[right] at (2.0,0.5) {$\bar{0}$};
	\node[right] at (2.0,1.5) {$\bar{0}$};
	\node[right, red] at (2.0,2.5) {$\bar{0}$};
	\node[right] at (2.0,3.5) {$\bar{0}$};
	\node[above, green] at (2.5,4.0) {$1$};
	\node[right] at (3.0,4.5) {$\bar{1}$};
	\node[right] at (3.0,5.5) {$\bar{1}$};
	\node[above] at (3.5,6.0) {$2$};
	\node[right, red] at (4.0,6.5) {$\bar{2}$};
	\node[right] at (4.0,7.5) {$\bar{2}$};
	\node[right] at (4.0,8.5) {$\bar{2}$};
	\node[above, green] at (4.5,9.0) {$3$};
	\node[above] at (5.5,9.0) {$3$};
	\node[right] at (6.0,9.5) {$\bar{3}$};
	\node[right] at (6.0,10.5) {$\bar{3}$};
	\end{tikzpicture}
	
	\caption{A doubly decorated reduced polyomino with its bounce path shown.}
	\label{fig:bounce-reduced-polyomino}
\end{figure}

The bounce word of a reduced polyomino is the sequence of labels attached to the steps of its bounce path.

\begin{definition}
	For a doubly decorated reduced polyomino $P \in \RP(m,n)^{\circ k, j}$ we let $\bounce(P)$ be the sum of the letters in its bounce word (disregarding bars), not counting the barred letters that lie in the same row of a decorated red valley, nor the unbarred letters that lie in the same column of a decorated green peak.
\end{definition}

Before introducing the third statistic, we need one more definition.

\begin{definition}
	Let $P \in \RP(m,n)$. For $1 \leq i < j \leq m+n$, we say that the pair $(i,j)$ is an \textit{inversion} if $i$-th letter of the area word of $P$ is the successor of the $j$-th letter of the area word of $P$ in the alphabet $\overline{\mathbb{N}}$.
\end{definition}

\begin{definition}
	For a reduced polyomino (possibly a doubly decorated one) $P \in \RP(m,n)^{\ast k, j}$ or $P \in \RP(m,n)^{\circ k, j}$ we let $\dinv(P)$ be the number of its inversions, which is not influenced by any decoration.
\end{definition}

Finally, we define $q,t$-enumerators for our sets as follows:

\begin{align*}
	\RP_{q,t}(m,n)^{\ast k, j} & \coloneqq \sum_{P \in \RP(m,n)^{\ast k, j} } q^{\dinv(P)} t^{\area(P)} \\
	\RP_{q,t}(m,n)^{\circ k, j} & \coloneqq \sum_{P \in \RP(m,n)^{\circ k, j} } q^{\area(P)} t^{\bounce(P)}
\end{align*} 

and the same for the other sets in which $r$ is specified. We have that these polynomials satisfy the same recursion as the symmetric function does.

\subsection{Recursion for $(\area, \bounce)$}

\begin{theorem}
	$\RP_{q,t}(m \backslash r, n)^{\circ k, j} = F_{m+1,r;n-j}^{(m+1-j,k)}$.
\end{theorem}

\begin{proof}
	It is enough to prove that $\RP_{q,t}(m \backslash r, n)^{\circ k, j}$ satisfy the recursion
	\begin{align*}
		\RP_{q,t}(m \backslash r, n)^{\circ k, j} & = t^{m-r-k+1} \sum_{w=0}^{r} \sum_{s=0}^{n} q^{\binom{w}{2}} \qbinom{r}{w}_q \qbinom{r+s-w-1}{s-w}_q \\ & \times \RP_{q,t}(n-1 \, \backslash \, s, \, m-r+1)^{\circ j-w, k}
	\end{align*}
	with 
	\[ \RP_{q,t}(m \backslash m+1, n)^{\circ k, j} = \delta_{k,0} q^{\binom{j}{2}} \qbinom{m+1}{j}_q \qbinom{m+n-j}{m}_q . \]
	\begin{figure}[!ht]
		\begin{center}
			\begin{tikzpicture}[scale=0.6]
			\draw[gray!60, thin] (0,0) grid (12,7);
			
			\filldraw[yellow, opacity=0.3] (0,0) -- (2,0) -- (2,3) -- (7,3) -- (7,4) -- (10,4) -- (10,5) -- (12,5) -- (12,6) -- (8,6) -- (8,5) -- (5,5) -- (5,4) -- (3,4) -- (3,3) -- (1,3) -- (1,2) -- (0,2) -- cycle;
			
			\draw[red, sharp <-sharp >, sharp < angle = 45, sharp > angle = 0, line width=1.6pt] (0,0) -- (0,2) -- (1,2) -- (1,3) -- (3,3) -- (3,4) -- (5,4) -- (5,5) -- (8,5) -- (8,6) -- (12,6) -- (12,7);
			
			\draw[green, sharp <-sharp >, sharp < angle = -45, sharp > angle = 0, line width=1.6pt] (0,0) -- (2,0) -- (2,3) -- (7,3) -- (7,4) -- (10,4) -- (10,5) -- (12,5) -- (12,7);
			
			\draw[red, sharp <-sharp >, sharp < angle = 45, sharp > angle = 0, dashed, line width=1.6pt] (0,0) -- (0,2) -- (1,2) -- (1,3) -- (3,3) -- (3,4) -- (5,4) -- (5,5) -- (8,5) -- (8,6) -- (12,6) -- (12,7);
			
			\draw[orange!60, line width = 1.6pt] (1.9,0.9) rectangle (12.1,7.1);
			
			\draw[dashed, opacity=0.6, thick] (0,0) -- (2,0) -- (2,3) -- (7,3) -- (7,5) -- (12,5) -- (12,7);
			
			\filldraw[fill=green]
			(2,3) circle (3pt)
			(7,4) circle (3pt);
			
			\filldraw[fill=red]
			(1,2) circle (3pt)
			(5,4) circle (3pt);
			
			\node[below] at (0.5,0) {$0$};
			\node[below] at (1.5,0) {$0$};
			\node[left] at (2,0.5) {$\bar{0}$};
			\node[left] at (2,1.5) {$\bar{0}$};
			\node[left, red] at (2,2.5) {$\bar{0}$};
			\node[below, green] at (2.5,3) {$1$};
			\node[below] at (3.5,3) {$1$};
			\node[below] at (4.5,3) {$1$};
			\node[below] at (5.5,3) {$1$};
			\node[below] at (6.5,3) {$1$};
			\node[left] at (7,3.5) {$\bar{1}$};
			\node[left, red] at (7,4.5) {$\bar{1}$};
			\node[below, green] at (7.5,5) {$2$};
			\node[below] at (8.5,5) {$2$};
			\node[below] at (9.5,5) {$2$};
			\node[below] at (10.5,5) {$2$};
			\node[below] at (11.5,5) {$2$};
			\node[left] at (12,5.5) {$\bar{2}$};
			\node[left] at (12,6.5) {$\bar{2}$};
			\end{tikzpicture}
		\end{center}
		
		\caption{One step of the recursion for reduced polyominoes.}
		\label{fig:polyorecursion}
	\end{figure}
	
	After one step of the recursion, the polyomino will be the one delimited by the orange rectangle, flipped along the line $x=y$. The bounce drops by $m-r-k+1$, because every unbarred letter decreases by one in the alphabet (so its value drops by $1$), except the $r-1$ $0$'s, that are just removed, and the $k$ letters corresponding to decorated green peaks, whose value actually decrease, but they are not counted while computing bounce and so they should be ignored.
	
	The factor $q^{\binom{w}{2}} \qbinom{r}{w}_q$ takes care of the area in the rows containing a vertical step immediately after one of the $w$ decorated red valleys with horizontal coordinate from $0$ to $r-1$. The factor $\qbinom{r+s-w-1}{s-w}_q$ takes care of the area in the remaining $s-w$ rows, where $s$ is the number of $\bar{0}$'s in the bounce word.
	
	Now, the green and the red path switch roles, and the $w$ decorations in the first $r-1$ columns disappear. The rest is easy to check.
\end{proof}

\subsection{Recursion for $(\dinv, \area)$}

\begin{theorem}
	$\RP_{q,t}(m \backslash r, n)^{\ast k, j} = F_{m+1,r;n-j}^{(m+1-j,k)}$.
\end{theorem}

\begin{proof}
	It is enough to prove that $\RP_{q,t}(m \backslash r, n)^{\ast k, j}$ satisfy the recursion
	\begin{align*}
		\RP_{q,t}(m \backslash r, n)^{\ast k, j} & = t^{m-r-k+1} \sum_{w=0}^{r} \sum_{s=0}^{n} q^{\binom{w}{2}} \qbinom{r}{w}_q \qbinom{r+s-w-1}{s-w}_q \\ & \times \RP_{q,t}(n-1 \, \backslash \, s, \, m-r+1)^{\ast j-w, k}
	\end{align*}
	with 
	\[ \RP_{q,t}(m \backslash m+1, n)^{\circ k, j} = \delta_{k,0} q^{\binom{j}{2}} \qbinom{m+1}{j}_q \qbinom{m+n-j}{m}_q . \]
	The recursive step consists in removing all the $0$'s, and going down by one step in the alphabet $0 < \bar{0} < 1 < \dots$. The area drops by $m+1-k-r$ (the number of unbarred, non decorated letters, minus the number of $0$'s). The factor $q^{\binom{w}{2}} \qbinom{r}{w}_q$ takes care of the inversions given by the $w$ decorated $\bar{0}$'s and the $0$'s. The factor $\qbinom{r+s-w-1}{s-w}_q$ takes care of the inversions between in the remaining $s-w$ $\bar{0}$'s and the $0$'s, where $s$ is the number of total $\bar{0}$'s.
	
	Now, barred and unbarred letters switch roles, and the $w$ decorations of rises of type $0 \bar{0}$ disappear. The rest is easy to check.
\end{proof}

\subsection{A statistics swapping bijection}

We also have a combinatorial proof of the identity \[ \RP_{q,t}(m,n)^{\ast k, j} = \RP_{q,t}(m,n)^{\circ k, j}. \]

\begin{theorem}
	\label{th:redzetamap}
	For $m \geq 0$, $n \geq 0$, $k \geq 0$, and $1 \leq r \leq m+1$, there is a bijection $\zeta \colon \RP(m \backslash r, n)^{\ast k, j} \rightarrow \RP(m \backslash r, n)^{\circ k, j}$ mapping the bistatistic $(\dinv, \area)$ to $(\area, \bounce)$.
\end{theorem}

\begin{proof}
	The map is essentially the same one described in \cite{Aval-DAdderio-Dukes-Hicks-LeBorgne-2014}*{Section~4}, adjusted to fit reduced polyominoes as in \cite{DAdderio-Iraci-polyominoes-2017}*{Theorem~3.9}. 
	
	Let us recall the definition of the $\zeta$ map. Pick a reduced polyomino with some decorated red valleys and green peaks and draw its bounce path; then, project the labels of the bounce path on both the red and the green path. Let us call \emph{bounce point} a vertex of the bounce path in which it changes direction. Now, build the area word of the image as follows: artificially start with a $0$, then pick the first bounce point on the red path, and write down the $0$'s and the $\bar{0}$'s as they appear going upwards along the red path up to that point (in this case, the relative order will always be with the $\bar{0}$ first, and all the $1$'s next). Then, go to the first bounce point on the green path, and insert the $1$'s after the correct number of $\bar{0}$'s, in the same relative order in which they appear going upwards to the previous bounce point. If a letter is decorated, keep the decoration. Now, move to the second bounce point on the red path, and repeat. See Figure~\ref{fig:zetamap} for an example.
	
	As proved in \cite{Aval-DAdderio-Dukes-Hicks-LeBorgne-2014}*{Section~4}, the result will be the area word of a $m \times n$ reduced polyomino. It is also proved that $\area$ is mapped to $\dinv$, since the squares of the starting reduced polyomino correspond to the inversions on the image.
	
	Red valleys are mapped into barred rises, because when reading the red path bottom to top, one reads the horizontal step first, which corresponds to an unbarred letter, and the vertical step next, which correspond to the next barred letter. Moreover, the decoration is kept on a letter with the same value. The same argument applies to green peaks being mapped to unbarred rises. This implies that $\bounce$ is mapped to $\area$.
	
	Furthermore, by construction one has that the number of $0$'s in the bounce word is equal to the number of $0$'s in the area word of the image polyomino (before adding the starting artificial one), since that area word is just an anagram of the bounce word of the preimage.
\end{proof}

\begin{figure}[!ht]
	\begin{center}
		\begin{tikzpicture}[scale=0.6]
		\draw[gray!60, thin] (0,0) grid (12,7);
		
		\filldraw[yellow, opacity=0.3] (0,0) -- (2,0) -- (2,3) -- (7,3) -- (7,4) -- (10,4) -- (10,5) -- (12,5) -- (12,6) -- (8,6) -- (8,5) -- (5,5) -- (5,4) -- (3,4) -- (3,3) -- (1,3) -- (1,2) -- (0,2) -- cycle;
		
		\draw[red, sharp <-sharp >, sharp < angle = 45, sharp > angle = 0, line width=1.6pt] (0,0) -- (0,2) -- (1,2) -- (1,3) -- (3,3) -- (3,4) -- (5,4) -- (5,5) -- (8,5) -- (8,6) -- (12,6) -- (12,7);
		
		\draw[green, sharp <-sharp >, sharp < angle = -45, sharp > angle = 0, line width=1.6pt] (0,0) -- (2,0) -- (2,3) -- (7,3) -- (7,4) -- (10,4) -- (10,5) -- (12,5) -- (12,7);
		
		\draw[red, sharp <-sharp >, sharp < angle = 45, sharp > angle = 0, dashed, line width=1.6pt] (0,0) -- (0,2) -- (1,2) -- (1,3) -- (3,3) -- (3,4) -- (5,4) -- (5,5) -- (8,5) -- (8,6) -- (12,6) -- (12,7);
		
		\draw[dashed, opacity=0.6, thick] (0,0) -- (2,0) -- (2,3) -- (7,3) -- (7,5) -- (12,5) -- (12,7);
		
		\filldraw[fill=green]
		(2,3) circle (3pt)
		(7,4) circle (3pt);
		
		\filldraw[fill=red]
		(1,2) circle (3pt)
		(5,4) circle (3pt);
		
		\node[above] at (0.5,0) {$0$};
		\node[above] at (1.5,0) {$0$};
		\node[right] at (2,0.5) {$\bar{0}$};
		\node[right] at (2,1.5) {$\bar{0}$};
		\node[right, red] at (2,2.5) {$\bar{0}$};
		\node[above, green] at (2.5,3) {$1$};
		\node[above] at (3.5,3) {$1$};
		\node[above] at (4.5,3) {$1$};
		\node[above] at (5.5,3) {$1$};
		\node[above] at (6.5,3) {$1$};
		\node[right] at (7,3.5) {$\bar{1}$};
		\node[right, red] at (7,4.5) {$\bar{1}$};
		\node[above, green] at (7.5,5) {$2$};
		\node[above] at (8.5,5) {$2$};
		\node[above] at (9.5,5) {$2$};
		\node[above] at (10.5,5) {$2$};
		\node[above] at (11.5,5) {$2$};
		\node[right] at (12,5.5) {$\bar{2}$};
		\node[right] at (12,6.5) {$\bar{2}$};
		
		\node[gray, left] at (0,0.5) {$\bar{0}$};
		\node[gray, left] at (0,1.5) {$\bar{0}$};
		\node[pink, left] at (1,2.5) {$\bar{0}$};
		\node[gray, below] at (0.5,2) {$0$};
		\node[gray, below] at (1.5,3) {$0$};
		\end{tikzpicture}
	\end{center}
	
	\caption{The first step needed to compute the $\zeta$ map. The final image will be the reduced polyomino with area word $0 \bar{0} \bar{0} 0 \!\! \stackrel{*}{{\color{red} \bar{0}}} \stackrel{*}{{\color{green} 1}} \!\! \bar{1} \!\! \stackrel{*}{{\color{green} 2}} \!\! \bar{2} 2 2 1 1 \!\! \stackrel{*}{{\color{red} \bar{1}}} \!\! 2 2 \bar{2} 1 1 0$}
	\label{fig:zetamap}
\end{figure}

\section{A statistics preserving bijection}

We have a combinatorial proof of the following theorem.

\begin{theorem}
	We have
	\[\RP_{q,t}(m \backslash r, n)^{\ast k, j}=\DDd_{q,t}(n-j,m+1 \backslash r)^{\ast k, \circ m+1-j},\]
	where $\DD(n-j,m+1 \backslash r)^{\ast k, \circ m+1-j}$ is the set of decorated Dyck paths of size $m+n-j+1$ with $m+1$ labels, $n-j$ zero valleys, $r$ $0$'s in the area word that are not zero valleys, $m+1-j$ decorated peaks, and $k$ decorated rises.
\end{theorem}

\begin{proof}
	We want to map $\RP(m \backslash r, n)^{\ast k, j}$ to $\DD(n-j,m+1 \backslash r)^{\ast k, \circ m+1-j}$, preserving the bistatistic $(\dinv, \area)$. We proceed as follows.
	
	Given the area word of such a polyomino, ignore the barred decorated letters. The remaining ones, disregarding bars, still form the area word of a Dyck path. This will be the actual path. If an unbarred letter is a decorated rise, then its image is still a rise, and we decorate it. We put zero labels on the steps corresponding to (non decorated) barred letters; those must be valleys, since there can't be a letter of strictly smaller value in the original area word of the polyomino. Next, we decorate all the peaks, except those whose corresponding letter of the area word of a polyomino is an unbarred letter followed by a decorated rise. Notice that all the steps that are neither zero valleys nor decorated peaks in the image must have this property (i.e. coming from an unbarred letter followed by a decorated rise). See Example~\ref{ex:polyo-to-pldp-bijection}.
	
	This maps obviously preserves the $\area$. If all the steps that are not zero valleys are decorated peaks, then the $\dinv$ can be counted only on those steps by looking at the primary on the left and the secondary on the right, and both those only come from zero valleys (on the same level on the left, one level lower on the right); those are exactly the contributions on the polyomino (counted on unbarred letters). If not, then each steps that neither a zero valley, nor a decorated peak adds a contribution given by its secondary $\dinv$ on the right and its primary on the left, which is exactly the contribution given on the $\dinv$ of the polyomino by the decorated barred letter that corresponds to that step on the image.
	
	To build the inverse map, one proceeds as follows. Given the area word of such a decorated Dyck path, one puts bars on letters corresponding to zero valleys, keeps the decorations on the rises (which are all unbarred, since all the barred letters are valleys and hence they can't be rises) then adds a decorated barred letter after each unbarred letter that does not correspond to a decorated peak (in particular one should do that if it doesn't correspond to a peak at all). It is easy to check that in this way one obtains the area word of a doubly decorated polyomino of the right size, and that this map is the inverse of the other.
\end{proof}

\begin{example}
	\label{ex:polyo-to-pldp-bijection}
	Let $0 \bar{0} \bar{0} 0 \!\! \stackrel{*}{{\color{red} \bar{0}}} \stackrel{*}{{\color{green} 1}} \!\! \bar{1} \!\! \stackrel{*}{{\color{green} 2}} \!\! \bar{2} 2 2 1 1 \!\! \stackrel{*}{{\color{red} \bar{1}}} \!\! 2 2 \bar{2} 1 1 0$ be the area word of a $12 \times 7$ doubly decorated reduced polyomino.
	
	We have two decorated barred rises, which we remove. The area word we get is $0 \bar{0} \bar{0} 0 \!\! \stackrel{*}{{\color{green} 1}} \!\! \bar{1} \!\! \stackrel{*}{{\color{green} 2}} \!\! \bar{2} 2 2 1 1 2 2 \bar{2} 1 1 0$, which (disregarding bars) is the area word of the Dyck path in Figure~\ref{fig:PLDP}.
	
	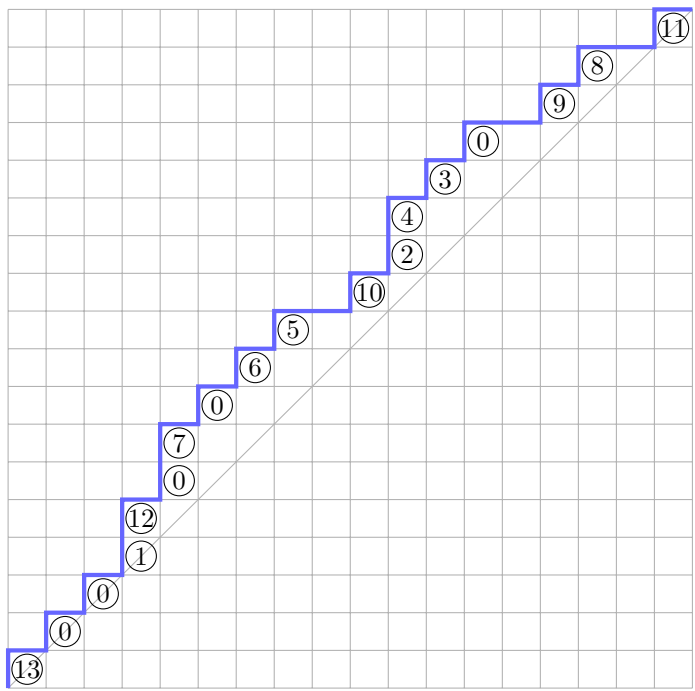
\begin{figure}[!ht]
		\centering
		\begin{tikzpicture}[scale=0.5]
		\draw[step=1.0, gray!60, thin] (0,0) grid (18,18);
		
		\draw[gray!60, thin] (0,0) -- (18,18);
		
		\draw[blue!60, line width=1.6pt] (0,0) -- (0,1) -- (1,1) -- (1,2) -- (2,2) -- (2,3) -- (3,3) -- (3,4) -- (3,5) -- (4,5) -- (4,6) -- (4,7) -- (5,7) -- (5,8) -- (6,8) -- (6,9) -- (7,9) -- (7,10) -- (8,10) -- (9,10) -- (9,11) -- (10,11) -- (10,12) -- (10,13) -- (11,13) -- (11,14) -- (12,14) -- (12,15) -- (13,15) -- (14,15) -- (14,16) -- (15,16) -- (15,17) -- (16,17) -- (17,17) -- (17,18) -- (18,18);

		\draw (0.5,0.5) circle (0.4 cm) node {$13$};
		\draw (1.5,1.5) circle (0.4 cm) node {$0$};
		\draw (2.5,2.5) circle (0.4 cm) node {$0$};
		\draw (3.5,3.5) circle (0.4 cm) node {$1$};
		\draw (3.5,4.5) circle (0.4 cm) node {$12$};
		\draw (4.5,5.5) circle (0.4 cm) node {$0$};
		\draw (4.5,6.5) circle (0.4 cm) node {$7$};
		\draw (5.5,7.5) circle (0.4 cm) node {$0$};
		\draw (6.5,8.5) circle (0.4 cm) node {$6$};
		\draw (7.5,9.5) circle (0.4 cm) node {$5$};
		\draw (9.5,10.5) circle (0.4 cm) node {$10$};
		\draw (10.5,11.5) circle (0.4 cm) node {$2$};
		\draw (10.5,12.5) circle (0.4 cm) node {$4$};
		\draw (11.5,13.5) circle (0.4 cm) node {$3$};
		\draw (12.5,14.5) circle (0.4 cm) node {$0$};
		\draw (14.5,15.5) circle (0.4 cm) node {$9$};
		\draw (15.5,16.5) circle (0.4 cm) node {$8$};
		\draw (17.5,17.5) circle (0.4 cm) node {$11$};
		
		\end{tikzpicture}
		\caption{The image of the reduced polyomino in Example~\ref{ex:polyo-to-pldp-bijection}}
		\label{fig:PLDP}
	\end{figure}
	
	We put zero labels on rows corresponding to barred letters (which are all valleys), and decorate all the peaks (i.e. labelling with big numbers in decreasing order) but the letters that were immediately before a decorated barred rise (which are the second $0$ and the third $1$, which will get labels $1$ and $2$).
\end{example}

\section{Open problems}

Recall the identity \eqref{eq:lemnablaEnk}, i.e.
\begin{equation}
F_{n,k;p}^{(d,\ell)}= \sum_{\gamma\vdash n+p-d}\left.(\mathbf{\Pi}^{-1}\nabla E_{n-\ell,k}[X])\right|_{X=MB_\gamma} \frac{\Pi_\gamma}{w_\gamma}e_{\ell}[B_\gamma]e_p[B_\gamma].
\end{equation}
Using \eqref{eq:Macdonald_reciprocity} and \eqref{eq:e_h_expansion}, it is not hard to show that for any $g\in \Lambda^{(n-\ell)}$,
\begin{equation}
\sum_{\gamma\vdash n+p-d}\left.(\mathbf{\Pi}^{-1}\nabla g[X])\right|_{X=MB_\gamma} \frac{\Pi_\gamma}{w_\gamma}e_{\ell}[B_\gamma]e_p[B_\gamma] = \left\< \Delta_{\phi \Delta_{e_{\ell}} \phi^{-1}(h_pe_{n-d})} \nabla g,h_{n-\ell}\right\>,
\end{equation}
where $\phi$ is the operator defined for any $f\in \Lambda$ by 
\begin{equation}
\phi f[X] :=f[MX].
\end{equation}
So we can write more compactly
\begin{equation} \label{eq:cpctFnkpdll}
F_{n,k;p}^{(d,\ell)}= \left\< \Delta_{\phi \Delta_{e_{\ell}} \phi^{-1}(h_pe_{n-d})} \nabla E_{n-\ell,k},h_{n-\ell}\right\>.
\end{equation}
For any composition $\alpha=(\alpha_1,\alpha_2,\dots,\alpha_{\ell(\alpha)})$, consider the symmetric functions $C_{\alpha}$ defined by
\begin{equation}
C_\alpha:=\mathbb{C}_{\alpha_1}\mathbb{C}_{\alpha_2}\cdots \mathbb{C}_{\alpha_{\ell(\alpha)}}(1)
\end{equation}
where the operators $\mathbb{C}_{m}$ are the ones appearing in \cite{Haglund-Morse-Zabrocki-2012}. Notice that the $C_\alpha$ are the essential ingredients of the \emph{compositional Shuffle conjecture} proved by Carlsson and Mellit in \cite{Carlsson-Mellit-ShuffleConj-2015}. Recall also from \cite{Haglund-Morse-Zabrocki-2012} the identity
\begin{equation} \label{eq:Calpha}
E_{n,k}=\mathop{\sum_{\alpha \vDash n}}_{\ell(\alpha)=k} C_\alpha.
\end{equation}

Lead by computer evidence, we risk the following conjecture.
\begin{conjecture}
	For any composition $\alpha\vDash n-\ell$,
	\begin{equation} \label{eq:openpbl}
	\left\< \Delta_{\phi \Delta_{e_{\ell}} \phi^{-1}(h_pe_{n-d})} \nabla C_\alpha,h_{n-\ell}\right\>\in \mathbb{N}[q,t].
	\end{equation}
\end{conjecture}
Observe that summing over compositions $\alpha\vDash n-\ell$ with $\ell(\alpha)=k$, and using \eqref{eq:Calpha}, we get precisely our polynomials \eqref{eq:cpctFnkpdll}. 

In support of this conjecture, it is not hard to see that for $p=0$ we have
\begin{align}
	\left\< \Delta_{\phi \Delta_{e_{\ell}} \phi^{-1}(h_0e_{n-d})} \nabla C_\alpha,h_{n-\ell}\right\> & = \left\< \Delta_{h_\ell e_{n-d-\ell}} \nabla C_\alpha,h_{n-\ell}\right\>\\
	& = \left\< \Delta_{h_\ell } \nabla C_\alpha,e_{n-d-\ell}h_{d}\right\>,
\end{align}
which are precisely the polynomials appearing in the \emph{$4$-variable Catalan conjecture} in \cite{Haglund-Remmel-Wilson-2015}, proved by Zabrocki in \cite{Zabrocki-4Catalan-2016} by showing that they have a ``compositional'' combinatorial interpretation. 

In fact, computer experiments show that the $t$-enumerator of the area of decorated Dyck paths of size $p+n$ with $p$ zero valleys, $\ell$ decorated rises, $d$ decorated peaks, and diagonal composition \emph{that ignores the rows of the decorated rises and of the zero valleys} $\alpha$ (so that $\alpha\vDash n-\ell$) equals the expression in \eqref{eq:openpbl} at $q=1$. Unfortunately the $q,t$-enumerator of (dinv,area) of the same set seems to be generally different from the expression in \eqref{eq:openpbl}.

Still, there might be a combinatorial interpretation of these polynomials that extends the one of the $4$-variable Catalan, which in turn would refine our Theorem~\ref{thm: recursion pld dinv}.

	
\bibliographystyle{amsalpha}
\bibliography{Biblebib}

\end{document}